\documentclass[a4paper, 10pt, conference, onecolumn]{ieeeconf}      

\IEEEoverridecommandlockouts                      
\usepackage[dvipdfmx]{graphicx}

\usepackage{comment} 
\usepackage{color} 
\usepackage{bm} 

\usepackage{amsmath,amssymb}
\usepackage{algorithm,algorithmic}
\usepackage{amsfonts}
\usepackage{physics}
\allowdisplaybreaks[4]
\usepackage{braket}
\usepackage{url}
\usepackage{lmodern}

\newcommand{\mcal}[1]{\mathcal{#1}}
\newcommand{\mb}[1]{\mathbb{#1}}

\newcommand{\pl}{\partial}

\newtheorem{theo}{Theorem}

\newcommand{\argmin}{\mathop{\rm argmin}\limits}

\title{\LARGE \bf
Memory-Limited Partially Observable Stochastic Control and its Mean-Field Control Approach
}

\author{Takehiro Tottori$^{1}$ and Tetsuya J. Kobayashi$^{1,2,3,4}$
\thanks{$^{1}$Department of Mathematical Informatics, Graduate School of Information Science and Technology, The University of Tokyo, Tokyo 113-8654, Japan}
\thanks{$^{2}$Department of Electrical Engineering and Information Systems, Graduate School of Engineering, The University of Tokyo, Tokyo 113-8654, Japan}
\thanks{$^{3}$Institute of Industrial Science, The University of Tokyo, Tokyo 153-8505, Japan}
\thanks{$^{4}$Universal Biology Institute, The University of Tokyo, Tokyo 113-8654, Japan}
}

\begin{document}
\maketitle
\thispagestyle{empty}
\pagestyle{empty}

\begin{abstract}
Control problems with incomplete information and memory limitation appear in many practical situations. Although partially observable stochastic control (POSC) is a conventional theoretical framework that considers the optimal control problem with incomplete information, it cannot consider memory limitation. Furthermore, POSC cannot be solved in practice except in the special cases. In order to address these issues, we propose an alternative theoretical framework, memory-limited POSC (ML-POSC). ML-POSC directly considers memory limitation as well as incomplete information, and it can be solved in practice by employing  the mathematical technique of the mean-field control theory. ML-POSC can generalize the LQG problem to include memory limitation. Because estimation and control are not clearly separated in the LQG problem with memory limitation, the Riccati equation is modified to the partially observable Riccati equation, which improves estimation as well as control. Furthermore, we demonstrate the effectiveness of ML-POSC to a non-LQG problem by comparing it with the local LQG approximation. 
\end{abstract}

\section{INTRODUCTION}
Control problems of systems with incomplete information and  memory limitation appear in many practical situations. 
These constraints become predominant, especially in designing the control of small devices \cite{fox_minimum-information_2016,fox_minimum-information_2016-1} and also in understanding the control mechanisms in biological systems \cite{li_iterative_2006,li_iterative_2007,nakamura_connection_2021,nakamura_optimal_2022,pezzotta_chemotaxis_2018,borra_optimal_2021} because their sensors are extremely noisy and their controllers can only have severely limited memories. 

Partially observable stochastic control (POSC) is a conventional theoretical framework that considers the optimal control problem with one of these constraints; the incomplete information of the system state \cite{bensoussan_stochastic_1992}. 
Because the controller of POSC cannot completely observe the state of the system, it determines the control based on the noisy observation history of the state. 
POSC can be solved in principle \cite{yong_stochastic_1999,nisio_stochastic_2015} by converting it to a completely observable stochastic control (COSC) of the posterior probability of the state because the posterior probability is the sufficient statistics of the observation history. 
The posterior probability and the optimal control are obtained by solving the Zakai equation and the Bellman equation, respectively.

However, POSC has three practical problems with respect to the implementation of the controller, which originate from the ignorance of the other constraint; the memory limitation of the controller. 
First, the controller designed by POSC should ideally have an infinite-dimensional memory to store and compute the posterior probability from the observation history. 
Second, the memory of the controller cannot have intrinsic stochasticity other than the observation noise to accurately compute the posterior probability via the Zakai equation. 
Third, POSC does not consider the cost originating from the memory update, which can be regarded as a cost for estimation. 
In the light of the dualistic roles played by estimation and control, considering only control cost by ignoring estimation cost is asymmetric. 
As a result, POSC is not practical for the control problem where the memory size, noise, and cost are not negligible. 
Therefore, we need an alternative theoretical framework considering memory limitation to circumvent these three problems. 

Furthermore, POSC has another crucial problem in obtaining the optimal state control by solving the Bellman equation. 
Because the posterior probability of the state is infinite-dimensional, POSC corresponds to an infinite-dimensional COSC. 
In the infinite-dimensional COSC, the Bellman equation becomes a functional differential equation, which should be solved to obtain the optimal state control. 
However, solving a functional differential equation is generally intractable even numerically. 

In this work, we propose an alternative theoretical framework to the conventional POSC, which can address the above-mentioned two issues. 
We call it memory-limited POSC (ML-POSC), in which memory limitation as well as incomplete information are directly accounted. 
The conventional POSC derives the Zakai equation without considering memory limitation. 
Then, the optimal state control is supposed to be derived by solving the Bellman equation (even though we do not have any practical way to do it). 
In contrast, ML-POSC first postulates the finite-dimensional and stochastic memory dynamics explicitly by taking the memory limitation into account and then jointly optimizes the memory dynamics and state control by considering the memory and control costs. 
As a result, unlike the conventional POSC, ML-POSC finds the optimal state control and the optimal memory dynamics under the given memory limitation. 
Furthermore, we show that the Bellman equation of ML-POSC can be reduced to the Hamilton-Jacobi-Bellman (HJB) equation by employing a trick in the mean-field control theory \cite{bensoussan_master_2015,bensoussan_interpretation_2017,bensoussan_mean_2021-1}. 
While the Bellman equation is a functional differential equation, the HJB equation is a partial differential equation. 
As a result, ML-POSC can be solved at least numerically.

The idea behind ML-POSC is closely related to that of the finite-state controller \cite{hansen_improved_1998,hansen_solving_1998,kaelbling_planning_1998,meuleau_solving_1999,meuleau_learning_1999,poupart_bounded_2003,amato_finite-state_2010}. 
The finite-state controller has been studied in Partially Observable Markov Decision Process (POMDP), which is the discrete time and state POSC.  
The finite-dimensional memory of ML-POSC can be regarded as an extension of the finite-state controller of POMDP to the continuous time and state setting. 
Nonetheless, the algorithms of the finite-state controller cannot be directly extended to our setting because they strongly depend on discreteness. 
ML-POSC resolves this problem by employing the mathematical technique of the mean-field control theory. 

In the Linear-Quadratic-Gaussian (LQG) problem of the conventional POSC, the Zakai equation and the Bellman equation are reduced to the Kalman filter and the Riccati equation, respectively \cite{bensoussan_stochastic_1992,bensoussan_estimation_2018}. 
Because the infinite-dimensional Zakai equation is reduced to the finite-dimensional Kalman filter, the LQG problem of the conventional POSC can also be discussed in terms of ML-POSC. 
We show that the Kalman filter corresponds to the optimal memory dynamics of ML-POSC. 
Moreover, ML-POSC can generalize the LQG problem to include the memory limitation such as the memory noise and cost. 
Because estimation and control are not clearly separated in the LQG problem with memory limitation, the Riccati equation for control is modified to include estimation, which is called the partially observable Riccati equation in this paper. 
We demonstrate that the partially observable Riccati equation is superior to the conventional Riccati equation in the LQG problem with memory limitation. 

Then, we investigate the potential effectiveness of ML-POSC to a non-LQG problem by comparing it with the local LQG approximation of the conventional POSC \cite{li_iterative_2006,li_iterative_2007}. 
In the local LQG approximation, the Zakai equation and the Bellman equation are locally approximated by the Kalman filter and the Riccati equation, respectively. 
Because the Bellman equation (functional differential equation) is reduced to the Riccati equation (ordinary differential equation), the local LQG approximation can be solved numerically. 
However, the performance of the local LQG approximation may be poor in a highly non-LQG problem because the local LQG approximation ignores non-LQG information. 
In contrast, ML-POSC reduces the Bellman equation (functional differential equation) to the HJB equation (partial differential equation) while maintaining non-LQG information. 
We demonstrate that ML-POSC can provide a better result than the local LQG approximation in a non-LQG problem.

This paper is organized as follows: 
In Sec. \ref{sec: Review of POSC}, we briefly review the conventional POSC. 
In Sec. \ref{sec: ML-POSC}, we formulate ML-POSC. 
In Sec. \ref{sec: MFC approach}, we propose the mean-field control approach to ML-POSC. 
In Sec. \ref{sec: LQG problem without memory limitation}, we investigate the LQG problem of the conventional POSC based on ML-POSC. 
In Sec. \ref{sec: LQG problem with memory limitation}, we generalize the LQG problem to include memory limitation. 
In Sec. \ref{sec: Numerical experiments}, we show the numerical experiments of a LQG problem with memory limitation and a non-LQG problem.  
In Sec. \ref{sec: Conclusion}, we conclude this paper. 

\section{REVIEW OF PARTIALLY OBSERVABLE STOCHASTIC CONTROL}\label{sec: Review of POSC}
In this section, we briefly review the conventional POSC \cite{nisio_stochastic_2015,bensoussan_mean_2021-1}. 

\subsection{Problem formulation}
In this subsection, we formulate the conventional POSC \cite{nisio_stochastic_2015,bensoussan_mean_2021-1}. 
The state $x_{t}\in\mb{R}^{d_{x}}$ and the observation $y_{t}\in\mb{R}^{d_{y}}$ at time $t\in[0,T]$ evolve by the following stochastic differential equations (SDEs): 
\begin{align}
	dx_{t}&=b(t,x_{t},u_{t})dt+\sigma(t,x_{t},u_{t})d\omega_{t},\label{eq: state SDE}\\
	dy_{t}&=h(t,x_{t})dt+\gamma(t)d\nu_{t},\label{eq: observation SDE}
\end{align}
where $x_{0}$ and $y_{0}$ obey $p_{0}(x_{0})$ and $p_{0}(y_{0})$, respectively, $\omega_{t}\in\mb{R}^{d_{\omega}}$ and $\nu_{t}\in\mb{R}^{d_{\nu}}$ are independent standard Wiener processes, and $u_{t}\in\mb{R}^{d_{u}}$ is the control. 
In POSC, because the controller cannot completely observe the state $x_{t}$, the control $u_{t}$ is determined based on the observation history $y_{0:t}:=\{y_{\tau}|\tau\in[0,t]\}$ as follows: 
\begin{align}
	u_{t}=u(t,y_{0:t}).\label{eq: control of POSC}
\end{align}

The objective function of POSC is given by the following expected cumulative cost function: 
\begin{align}	
	J[u]:=\mb{E}_{p(x_{0:T},y_{0:T};u)}\left[\int_{0}^{T}f(t,x_{t},u_{t})dt+g(x_{T})\right],
	\label{eq: OF of POSC}
\end{align}
where $f$ is the cost function, $g$ is the terminal cost function, $p(x_{0:T},y_{0:T};u)$ is the probability of $x_{0:T}$ and $y_{0:T}$ given $u$ as a parameter, and $\mb{E}_{p}\left[\cdot\right]$ is the expectation with respect to the probability $p$. 

POSC is the problem to find the optimal control function $u^{*}$ that minimizes the objective function $J[u]$ as follows: 
\begin{align}	
	u^{*}:=\argmin_{u}J[u].
	\label{eq: POSC}
\end{align}

\subsection{Derivation of optimal control function}
In this subsection, we briefly review the derivation of the optimal control function of the conventional POSC \cite{nisio_stochastic_2015,bensoussan_mean_2021-1}. 
We first define the unnormalized posterior probability density function $q_{t}(x):=p(x_{t}=x,y_{0:t})$. 
We omit $y_{0:t}$ for the notational simplicity. 
$q_{t}(x)$ obeys the following Zakai equation: 
\begin{align}
	dq_{t}(x)=\mcal{L}^{\dag}q_{t}(x)+q_{t}(x)h^{T}(t,x)(\gamma(t)\gamma^{T}(t))^{-1}dy_{t}, 
	\label{eq: Zakai eq}
\end{align}
where $q_{0}(x)=p_{0}(x)p_{0}(y)$, and $\mcal{L}^{\dag}$ is the forward diffusion operator, which is defined by
\begin{align}
	\mcal{L}^{\dag}q(x)&:=-\sum_{i=1}^{d_{x}}\frac{\pl (b_{i}(t,x,u)q(x))}{\pl x_{i}}
	+\frac{1}{2}\sum_{i,j=1}^{d_{x}}\frac{\pl^{2}(D_{ij}(t,x,u)q(x))}{\pl x_{i}\pl x_{j}}, 
	\label{eq: forward diffusion operator}
\end{align}
where $D(t,x,u):=\sigma(t,x,u)\sigma^{T}(t,x,u)$. 
Then, the objective function (\ref{eq: OF of POSC}) can be calculated as follows: 
\begin{align}
	J[u]=\mb{E}_{p(q_{0:T};u)}\left[\int_{0}^{T}\bar{f}(t,q_{t},u_{t})dt+\bar{g}(q_{T})\right], 
	\label{eq: OF of Zakai eq}
\end{align}
where $\bar{f}(t,q,u):=\mb{E}_{q(x)}\left[f(t,x,u)\right]$ and $\bar{g}(q):=\mb{E}_{q(x)}\left[g(x)\right]$. 
From (\ref{eq: Zakai eq}) and (\ref{eq: OF of Zakai eq}), POSC is converted into a completely observable stochastic control (COSC) of $q_{t}$. 
As a result, POSC can be approached in the similar way as COSC, and the optimal control function is given by the following theorem: 

\begin{theo}
The optimal control function of POSC is given by
\begin{align}
	u^{*}(t,q)=\argmin_{u}\mb{E}_{q(x)}\left[H\left(t,x,u,\frac{\delta V(t,q)}{\delta q}(x)\right)\right], 
	\label{eq: optimal control of POSC}
\end{align}
where $H$ is the Hamiltonian, which is defined by
\begin{align}
	H\left(t,x,u,\frac{\delta V(t,q)}{\delta q}(x)\right):=f(t,x,u)+\mcal{L}\frac{\delta V(t,q)}{\delta q}(x), 
	\label{eq: Hamiltonian}
\end{align}
where $\mcal{L}$ is the backward diffusion operator, which is defined by 
\begin{align}
	\mcal{L}q(x)&:=\sum_{i=1}^{d_{x}}b_{i}(t,x,u)\frac{\pl q(x)}{\pl x_{i}}
	+\frac{1}{2}\sum_{i,j=1}^{d_{x}}D_{ij}(t,x,u)\frac{\pl^{2} q(x)}{\pl x_{i}\pl x_{j}}.
	\label{eq: backward diffusion operator}
\end{align}
We note that $\mcal{L}$ is the conjugate of $\mcal{L}^{\dag}$. 
$V(t,q)$ is the value function, which is the solution of the following Bellman equation: 
\begin{align}
	&-\frac{\pl V(t,q)}{\pl t}=\mb{E}_{q(x)}\left[H\left(t,x,u^{*},\frac{\delta V(t,q)}{\delta q}(x)\right)\right]\nonumber\\
	&+\frac{1}{2}\mb{E}_{q(x)q(x')}\left[\frac{\delta}{\delta q}\frac{\delta V(t,q)}{\delta q}(x,x')h^{T}(t,x)(\gamma(t)\gamma^{T}(t))^{-1}h(t,x')\right], 
	\label{eq: Bellman eq of POSC}
\end{align}
where $V(T,q)=\mb{E}_{q(x)}\left[g(x)\right]$. 
\end{theo}

\begin{proof}
The proof is shown in \cite{nisio_stochastic_2015,bensoussan_mean_2021-1}. 
\end{proof}

The optimal control function $u^{*}(t,q)$ is obtained by solving the Bellman equation (\ref{eq: Bellman eq of POSC}). 
The controller determines the optimal control $u_{t}^{*}=u^{*}(t,q_{t})$ based on the posterior probability $q_{t}$. 
The posterior probability $q_{t}$ is obtained by solving the Zakai equation (\ref{eq: Zakai eq}). 
As a result, POSC can be solved in principle. 

However, POSC has three practical problems with respect to the memory of the controller. 
First, the controller should have an infinite-dimensional memory to store and compute the posterior probability $q_{t}$ from the observation history $y_{0:t}$. 
Second, the memory of the controller cannot have intrinsic stochasticity other than the observation $dy_{t}$ to accurately compute the posterior probability $q_{t}$ via the Zakai equation (\ref{eq: Zakai eq}). 
Third, POSC does not consider the cost originating from the memory update, which can be regarded as a cost for estimation. 
In light of the dualistic roles played by estimation and control, considering only control cost by ignoring estimation cost is asymmetric. 
As a result, POCS is not practical for the control problem where the memory size, noise, and cost are not negligible. 

Furthermore, POSC has another crucial problem in obtaining the optimal control function $u^{*}(t,q)$ by solving the Bellman equation (\ref{eq: Bellman eq of POSC}). 
Because the posterior probability $q$ is infinite-dimensional, the associated Bellman equation (\ref{eq: Bellman eq of POSC}) becomes a functional differential equation. 
However, solving a functional differential equation is generally intractable even numerically. 
As a result, POCS cannot be solved in practice. 

\section{MEMORY-LIMITED PARTIALLY OBSERVABLE STOCHASTIC CONTROL}\label{sec: ML-POSC}
In order to address the above-mentioned problems, we propose an alternative theoretical framework to the conventional POSC, ML-POSC. 
In this section, we formulate ML-POSC. 

\subsection{Problem formulation}
In this subsection, we formulate ML-POSC. 
ML-POSC determines the control $u_{t}$ based on the finite-dimensional memory $z_{t}\in\mb{R}^{d_{z}}$ as follows: 
\begin{align}
	u_{t}=u(t,z_{t}).\label{eq: control of ML-POSC}
\end{align}
The memory dimension $d_{z}$ is determined not by the optimization but by the prescribed memory limitation of the controller to be used. 
Comparing (\ref{eq: control of POSC}) and (\ref{eq: control of ML-POSC}), the memory $z_{t}$ can be interpreted as the compression of the observation history $y_{0:t}$. 
While the conventional POSC compresses the observation history $y_{0:t}$ into the infinite-dimensional posterior probability $q_{t}$, ML-POSC compresses it into the finite-dimensional memory $z_{t}$. 

ML-POSC formulates the memory dynamics with the following SDE: 
\begin{align}
	dz_{t}=c(t,z_{t},v_{t})dt+\kappa(t,z_{t},v_{t})dy_{t}+\eta(t,z_{t},v_{t})d\xi_{t},
	\label{eq: memory SDE}
\end{align}
where $z_{0}$ obeys $p_{0}(z_{0})$, $\xi_{t}\in\mb{R}^{d_{\xi}}$ is the standard Wiener process, and $v_{t}=v(t,z_{t})\in\mb{R}^{d_{v}}$ is the control for the memory dynamics. 
This memory dynamics has three important properties: 
(i) Because it depends on the observation $dy_{t}$, the memory $z_{t}$ can be interpreted as the compression of the observation history $y_{0:t}$. 
(ii)  Because it depends on the standard Wiener process $d\xi_{t}$, ML-POSC can consider the memory noise explicitly. 
(iii) Because it depends on the control $v_{t}$, it can be optimized through the control $v_{t}$. 

The objective function of ML-POSC is given by the following expected cumulative cost function: 
\begin{align}	
	J[u,v]:=\mb{E}_{p(x_{0:T},y_{0:T},z_{0:T};u,v)}\left[\int_{0}^{T}f(t,x_{t},u_{t},v_{t})dt+g(x_{T})\right]. 
	\label{eq: OF of ML-POSC}
\end{align}
Because the cost function $f$ depends on the memory control $v_{t}$ as well as the state control $u_{t}$, 
ML-POSC can consider the memory control cost (state estimation cost) as well as the state control cost explicitly.

ML-POSC optimizes the state control function $u$ and the memory control function $v$ based on the objective function $J[u,v]$ as follows: 
\begin{align}	
	u^{*},v^{*}:=\argmin_{u,v}J[u,v].
	\label{eq: ML-POSC}
\end{align}

ML-POSC first postulates the finite-dimensional and stochastic memory dynamics explicitly and then jointly optimizes the state and memory control function by considering the state and memory control cost. 
As a result, unlike the conventional POSC, ML-POSC can consider memory limitation as well as incomplete information. 

\subsection{Problem reformulation}
Although the formulation of ML-POSC in the previous subsection clarifies its relationship with that of the conventional POSC,  it is inconvenient for further mathematical investigations. 
In order to resolve this problem, we reformulate ML-POSC in this subsection. 
The formulation in this subsection is simpler and more general than that in the previous subsection. 

We first define the extended state $s_{t}$ as follows: 
\begin{align}
	s_{t}:=\left(\begin{array}{c}
		x_{t}\\ z_{t}\\
	\end{array}\right)\in\mb{R}^{d_{s}}, 
	\label{eq: extended state}
\end{align}
where $d_{s}=d_{x}+d_{z}$. 
The extended state $s_{t}$ evolves by the following SDE: 
\begin{align}
	ds_{t}=\tilde{b}(t,s_{t},\tilde{u}_{t})dt+\tilde{\sigma}(t,s_{t},\tilde{u}_{t})d\tilde{\omega}_{t}, 
	\label{eq: extended state SDE}
\end{align}
where $s_{0}$ obeys $p_{0}(s_{0})$, $\tilde{\omega}_{t}\in\mb{R}^{d_{\tilde{\omega}}}$ is the standard Wiener process, and $\tilde{u}_{t}\in\mb{R}^{d_{\tilde{u}}}$ is the control.
ML-POSC determines the control $\tilde{u}_{t}\in\mb{R}^{d_{\tilde{u}}}$ based solely on the memory $z_{t}$ as follows: 
\begin{align}
	\tilde{u}_{t}=\tilde{u}(t,z_{t}).\label{eq: control of GML-POSC}
\end{align}
The extended state SDE (\ref{eq: extended state SDE}) includes the previous state, observation, and memory SDEs (\ref{eq: state SDE}), (\ref{eq: observation SDE}), (\ref{eq: memory SDE}) as a special case because they can be represented as follows: 
\begin{align}
	&ds_{t}=\left(\begin{array}{c}
		b(t,x_{t},u_{t})\\
		c(t,z_{t},v_{t})+\kappa(t,z_{t},v_{t})h(t,x_{t})\\
	\end{array}\right)dt\nonumber\\
	&+\left(\begin{array}{ccc}
		\sigma(t,x_{t},u_{t})&O&O\\
		O&\kappa(t,z_{t},v_{t})\gamma(t)&\eta(t,z_{t},v_{t})\\
	\end{array}\right)
	\left(\begin{array}{c}
		d\omega_{t}\\
		d\nu_{t}\\
		d\xi_{t}\\
	\end{array}\right),
\end{align}
where $p_{0}(s_{0})=p_{0}(x_{0})p_{0}(z_{0})$.

The objective function of ML-POSC is given by the following expected cumulative cost function: 
\begin{align}	
	J[\tilde{u}]:=\mb{E}_{p(s_{0:T};\tilde{u})}\left[\int_{0}^{T}\tilde{f}(t,s_{t},\tilde{u}_{t})dt+\tilde{g}(s_{T})\right],
	\label{eq: OF of GML-POSC}
\end{align} 
where $\tilde{f}$ is the cost function, and $\tilde{g}$ is the terminal cost function.  
It is obvious that this objective function (\ref{eq: OF of GML-POSC}) is more general than the previous one (\ref{eq: OF of ML-POSC}). 

ML-POSC is the problem to find the optimal control function $\tilde{u}^{*}$ that minimizes the objective function $J[\tilde{u}]$ as follows: 
\begin{align}	
	\tilde{u}^{*}:=\argmin_{\tilde{u}}J[\tilde{u}].
	\label{eq: GML-POSC}
\end{align}

In the following section, we mainly consider the formulation of this subsection rather than that of the previous subsection because it is simpler and more general. 
Moreover, we omit $\tilde{\cdot}$ for the notational simplicity.

\section{MEAN-FIELD CONTROL APPROACH}\label{sec: MFC approach}
If the control $u_{t}$ is determined based on the extended state $s_{t}$, i.e., $u_{t}=u(t,s_{t})$, ML-POSC is the same with the COSC of the extended state $s_{t}$, and it can be solved by the conventional COSC approach \cite{yong_stochastic_1999}. 
However,  because ML-POSC determines the control $u_{t}$ based solely on the memory $z_{t}$, i.e., $u_{t}=u(t,z_{t})$, ML-POSC cannot be solved in the similar way as COSC. 
In order to solve ML-POSC, we propose the mean-field control approach in this section. 
Because the mean-field control approach is more general than the COSC approach, it can solve COSC and ML-POSC in a unified way.  

\subsection{Derivation of optimal control function}
In this subsection, we propose the mean-field control approach to ML-POSC. 
We first show that ML-POSC can be converted into a deterministic control of the probability density function, which is similar with the conventional POSC \cite{nisio_stochastic_2015,bensoussan_mean_2021-1}. 
This approach is also used in the mean-field control \cite{bensoussan_master_2015,bensoussan_interpretation_2017,lauriere_dynamic_2016,pham_bellman_2018}.
The extended state SDE (\ref{eq: extended state SDE}) can be converted into the following Fokker-Planck (FP) equation: 
\begin{align}
	\frac{\pl p_{t}(s)}{\pl t}=\mcal{L}^{\dagger}p_{t}(s), 
	\label{eq: FP eq}
\end{align}
where the initial condition is given by $p_{0}(s)$, and the forward diffusion operator $\mcal{L}^{\dagger}$ is defined by (\ref{eq: forward diffusion operator}). 
The objective function of ML-POSC (\ref{eq: OF of GML-POSC}) can be calculated as follows: 
\begin{align}
	J[u]=\int_{0}^{T}\bar{f}(t,p_{t},u_{t})dt+\bar{g}(p_{T}), 
	\label{eq: OF of FP eq}
\end{align}
where $\bar{f}(t,p,u):=\mb{E}_{p(s)}[f(t,s,u)]$ and 
$\bar{g}(p):=\mb{E}_{p(s)}[g(s)]$. 
From (\ref{eq: FP eq}) and (\ref{eq: OF of FP eq}), ML-POSC is converted into a deterministic control of $p_{t}$. 
As a result, ML-POSC can be approached in the similar way as deterministic control, and the optimal control function is given by the following theorem: 

\begin{theo}
\label{theo: optimal control of GML-POSC based on Bellman eq}
The optimal control function of ML-POSC is given by 
\begin{align}
	u^{*}(t,z)=\argmin_{u}\mb{E}_{p_{t}(x|z)}\left[H\left(t,s,u,\frac{\delta V(t,p_{t})}{\delta p}(s)\right)\right], 
	\label{eq: optimal control of GML-POSC based on Bellman eq}
\end{align}
where $H$ is the Hamiltonian (\ref{eq: Hamiltonian}), 
$p_{t}(x|z)=p_{t}(s)/\int p_{t}(s)dx$ is the conditional probability density function of the state $x$ given the memory $z$, 
$p_{t}(s)$ is the solution of the FP equation (\ref{eq: FP eq}), 
and $V(t,p)$ is the solution of the following Bellman equation: 
\begin{align}
	-\frac{\pl V(t,p)}{\pl t}=\mb{E}_{p(s)}\left[H\left(t,s,u^{*},\frac{\delta V(t,p)}{\delta p}(s)\right)\right], 
	\label{eq: Bellman eq of GML-POSC}
\end{align}
where $V(T,p)=\mb{E}_{p(s)}[g(s)]$. 
\end{theo}

\begin{proof}
The proof is shown in Appendix A. 
\end{proof}

The controller of ML-POSC determines the optimal control $u_{t}^{*}=u^{*}(t,z_{t})$ based on the memory $z_{t}$, not the posterior probability $q_{t}$. 
Therefore, ML-POSC can consider memory limitation as well as incomplete information. 

However, because the Bellman equation (\ref{eq: Bellman eq of GML-POSC}) is a functional differential equation, it cannot be solved even numerically, which is the same problem as the conventional POSC. 
We resolve this problem by employing a mathematical technique of the mean-field control theory \cite{bensoussan_master_2015,bensoussan_interpretation_2017} as follows:

\begin{theo}\label{theo: optimal control of GML-POSC}
The optimal control function of ML-POSC is given by
\begin{align}
	u^{*}(t,z)=\argmin_{u}\mb{E}_{p_{t}(x|z)}\left[H\left(t,s,u,w(t,s)\right)\right], 
	\label{eq: optimal control of GML-POSC}
\end{align}
where $H$ is the Hamiltonian (\ref{eq: Hamiltonian}), 
$p_{t}(x|z)=p_{t}(s)/\int p_{t}(s)dx$ is the conditional probability density function of the state $x$ given the memory $z$, 
$p_{t}(s)$ is the solution of the FP equation (\ref{eq: FP eq}), 
and $w(t,s)$ is the solution of the following Hamilton-Jacobi-Bellman (HJB) equation: 
\begin{align}
	-\frac{\pl w(t,s)}{\pl t}=H\left(t,s,u^{*},w(t,s)\right),
	\label{eq: HJB eq}
\end{align}
where $w(T,s)=g(s)$. 
\end{theo}

\begin{proof}
The proof is shown in Appendix B. 
\end{proof}

While the Bellman equation (\ref{eq: Bellman eq of GML-POSC}) is a functional differential equation, the HJB equation (\ref{eq: HJB eq}) is a partial differential equation. 
As a result, unlike the conventional POSC, ML-POSC can be solved in practice. 

We note that the mean-field control technique is also applicable to the conventional POSC, and we obtain the HJB equation of the conventional POSC \cite{bensoussan_mean_2021-1}. 
However, the HJB equation of the conventional POSC is not closed by a partial differential equation because of the last term of the Bellman equation (\ref{eq: Bellman eq of POSC}). 
As a result, the mean-field control technique is not effective to the conventional POSC except in a special case \cite{bensoussan_mean_2021-1}. 

In the conventional POSC, the state estimation (memory control) and the state control are clearly separated. 
As a result, the state estimation and the state control are optimized by the Zakai equation (\ref{eq: Zakai eq}) and the Bellman equation (\ref{eq: Bellman eq of POSC}), respectively. 
In contrast, because ML-POSC considers memory limitation as well as incomplete information, the state estimation and the state control are not clearly separated. 
As a result, ML-POSC jointly optimizes the state estimation and the state control based on the FP equation (\ref{eq: FP eq}) and the HJB equation (\ref{eq: HJB eq}). 

\subsection{Comparison with completely observable stochastic control}
In this subsection, we show the similarities and the differences between ML-POSC and the COSC of the extended state. 
While ML-POSC determines the control $u_{t}$ based solely on the memory $z_{t}$, i.e., $u_{t}=u(t,z_{t})$, the COSC of the extended state determines the control $u_{t}$ based on the extended state $s_{t}$, i.e., $u_{t}=u(t,s_{t})$. 
The optimal control function of the COSC of the extended state is given by the following theorem: 

\begin{theo}\label{theo: optimal control of COSC}
The optimal control function of the COSC of the extended state is given by
\begin{align}
	u^{*}(t,s)=\argmin_{u}H\left(t,s,u,w(t,s)\right),
	\label{eq: optimal control of COSC}
\end{align}
where $H$ is the Hamiltonian (\ref{eq: Hamiltonian}), and $w(t,s)$ is the solution of the HJB equation (\ref{eq: HJB eq}). 
\end{theo}

\begin{proof}
The conventional proof is shown in \cite{yong_stochastic_1999}. 
We note that it can be proved in the similar way as ML-POSC, which is shown in Appendix C. 
\end{proof}

Although the HJB equation (\ref{eq: HJB eq}) is the same between ML-POSC and COSC, the optimal control function is different. 
While the optimal control function of COSC is given by the minimization of the Hamiltonian (\ref{eq: optimal control of COSC}), 
that of ML-POSC is given by the minimization of the conditional expected Hamiltonian (\ref{eq: optimal control of GML-POSC}). 
This is reasonable because the controller of ML-POSC needs to estimate the state from the memory. 

\subsection{Numerical algorithm}
In this subsection, we briefly explain a numerical algorithm to obtain the optimal control function of ML-POSC (\ref{eq: optimal control of GML-POSC}). 
Because the optimal control function of COSC (\ref{eq: optimal control of COSC}) depends only on the backward HJB equation (\ref{eq: HJB eq}), 
it can be obtained by solving the HJB equation backward from the terminal condition \cite{yong_stochastic_1999,kushner_numerical_1992,fleming_controlled_2006}.
In contrast, because the optimal control function of ML-POSC (\ref{eq: optimal control of GML-POSC}) depends on the forward FP equation (\ref{eq: FP eq}) as well as the backward HJB equation (\ref{eq: HJB eq}), it cannot be obtained in the similar way as COSC. 
Because the backward HJB equation depends on the forward FP equation through the optimal control function of ML-POSC, the HJB equation cannot be solved backward from the terminal condition. 
As a result, ML-POSC needs to solve the system of HJB-FP equations. 

The system of HJB-FP equations also appears in the mean-field game and control \cite{bensoussan_mean_2013,carmona_probabilistic_2018,carmona_probabilistic_2018-1}, and numerous numerical algorithms have been developed \cite{achdou_finite_2013,achdou_mean_2020,lauriere_numerical_2021}. 
Therefore, unlike the conventional POSC, ML-POSC can be solved in practice by using these algorithms.   
Furthermore, unlike the mean-field game and control, the coupling of HJB-FP equations is limited to the optimal control function in ML-POSC. 
By exploiting this property, more efficient algorithms may be proposed in ML-POSC \cite{tottori_notitle_2022}. 

In this paper, we use the forward-backward sweep method (fixed-point iteration method) to obtain the optimal control function of ML-POSC \cite{lauriere_numerical_2021,carlini_semi-lagrangian_2013,carlini_fully_2014,carlini_semi-lagrangian_2015,tottori_notitle_2022}, which is one of the most basic algorithms for the system of HJB-FP equations. 
The forward-backward sweep method computes the forward FP equation (\ref{eq: FP eq}) and the backward HJB equation (\ref{eq: HJB eq}) alternately. 
In the mean-field game and control, the convergence of the forward-backward sweep method is not guaranteed. 
In contrast, it is guaranteed in ML-POSC because the coupling of HJB-FP equations is limited to the optimal control function in ML-POSC \cite{tottori_notitle_2022}. 

\section{LINEAR-QUADRATIC-GAUSSIAN PROBLEM WITHOUT MEMORY LIMITATION}\label{sec: LQG problem without memory limitation}
In the Linear-Quadratic-Gaussian (LQG) problem of the conventional POSC, the Zakai equation (\ref{eq: Zakai eq}) and the Bellman equation (\ref{eq: Bellman eq of POSC}) are reduced to the Kalman filter and the Riccati equation, respectively \cite{bensoussan_stochastic_1992,bensoussan_estimation_2018}. 
Because the infinite-dimensional Zakai equation is reduced to the finite-dimensional Kalman filter, the LQG problem of the conventional POSC can also be discussed in terms of ML-POSC. 
In this section, we briefly review the LQG problem of the conventional POSC, and then reproduce the Kalman filter and the Riccati equation from the viewpoint of ML-POSC. 
We note that the LQG problem of the conventional POSC cannot consider the memory limitation such as the memory noise and cost whereas ML-POSC can. 

\subsection{Review of partially observable stochastic control}
In this subsection, we briefly review the LQG problem of the conventional POSC \cite{bensoussan_stochastic_1992,bensoussan_estimation_2018}. 
The state $x_{t}\in\mb{R}^{d_{x}}$ and the observation $y_{t}\in\mb{R}^{d_{y}}$ at time $t\in[0,T]$ evolve by the following SDEs: 
\begin{align}
	dx_{t}&=\left(A(t)x_{t}+B(t)u_{t}\right)dt+\sigma(t)d\omega_{t},\label{eq: state SDE LQG-POSC}\\
	dy_{t}&=H(t)x_{t}dt+\gamma(t)d\nu_{t},\label{eq: observation SDE LQG-POSC}
\end{align}
where $x_{0}$ obeys the Gaussian distribution $p_{0}(x_{0})=\mcal{N}\left(x_{0}\left|\mu_{x,0},\Sigma_{xx,0}\right.\right)$, $y_{0}$ is an arbitrary real vector, $\omega_{t}\in\mb{R}^{d_{\omega}}$ and $\nu_{t}\in\mb{R}^{d_{\nu}}$ are independent standard Wiener processes, and $u_{t}=u(t,y_{0:t})\in\mb{R}^{d_{u}}$ is the control. 
The objective function is given by the following expected cumulative cost function: 
\begin{align}	
	J[u]:=\mb{E}_{p(x_{0:T},y_{0:T};u)}\left[\int_{0}^{T}\left(x_{t}^{T}Q(t)x_{t}+u_{t}^{T}R(t)u_{t}\right)dt+x_{T}^{T}Px_{T}\right],
	\label{eq: OF of POSC LQG-POSC}
\end{align}
where $Q(t)\succeq O$, $R(t)\succ O$, and $P\succeq O$. 
The LQG problem of the conventional POSC is to find the optimal control function $u^{*}$ that minimizes the objective function $J[u]$ as follows: 
\begin{align}	
	u^{*}:=\argmin_{u}J[u].
	\label{eq: POSC LQG-POSC}
\end{align}

In the LQG problem of the conventional POSC, the posterior probability is given by the Gaussian distribution $p(x_{t}|y_{0:t})=\mcal{N}(x_{t}|\check{\mu}(t),\check{\Sigma}(t))$, and $u_{t}=u(t,y_{0:t})$ is reduced to $u_{t}=u(t,\check{\mu}_{t})$ without loss of performance.  

\begin{theo}\label{theo: optimal control of POSC in LQG-POSC}
In LQG problem of the conventional POSC, the optimal control function is given by
\begin{align}
	u^{*}(t,\check{\mu})=-R^{-1}B^{T}\Psi\check{\mu},
\end{align}
where $\check{\mu}(t)$ and $\check{\Sigma}(t)$ are the solutions of the following Kalman filter: 
\begin{align}
	&d\check{\mu}=\left(A-BR^{-1}B^{T}\Psi\right)\check{\mu}dt+\check{\Sigma}H^{T}(\gamma\gamma^{T})^{-1}\left(dy_{t}-H\check{\mu}dt\right),\label{eq: Kalman filter mu}\\
	&\frac{d\check{\Sigma}}{dt}=\sigma\sigma^{T}+A\check{\Sigma}+\check{\Sigma}A^{T}-\check{\Sigma}H^{T}(\gamma\gamma^{T})^{-1}H\check{\Sigma},\label{eq: Kalman filter Sigma}
\end{align}
where $\check{\mu}(0)=\mu_{x,0}$ and $\check{\Sigma}(0)=\Sigma_{xx,0}$. 
$\Psi(t)$ is the solution of the following Riccati equation: 
\begin{align}
	-\frac{d\Psi}{dt}&=Q+A^{T}\Psi+\Psi A -\Psi BR^{-1}B^{T}\Psi,\label{eq: ODE of Psi LQG-POSC}
\end{align}
where $\Psi(T)=P$. 
\end{theo}

\begin{proof}
The proof is shown in \cite{bensoussan_stochastic_1992,bensoussan_estimation_2018}.
\end{proof}

In the LQG problem of the conventional POSC, the Zakai equation (\ref{eq: Zakai eq}) and the Bellman equation (\ref{eq: Bellman eq of POSC}) are reduced to the Kalman filter (\ref{eq: Kalman filter mu}), (\ref{eq: Kalman filter Sigma}) and the Riccati equation (\ref{eq: ODE of Psi LQG-POSC}), respectively. 

\subsection{Reproduction by memory-limited partially observable stochastic control}
Because the infinite-dimensional Zakai equation (\ref{eq: Zakai eq}) is reduced to the finite-dimensional Kalman filter (\ref{eq: Kalman filter mu}), (\ref{eq: Kalman filter Sigma}), the LQG problem of the conventional POSC can also be discussed in terms of ML-POSC. 
In this subsection, we reproduce the Kalman filter (\ref{eq: Kalman filter mu}), (\ref{eq: Kalman filter Sigma}) and the Riccati equation (\ref{eq: ODE of Psi LQG-POSC}) from the viewpoint of ML-POSC. 

ML-POSC defines the finite-dimensional memory $z_{t}\in\mb{R}^{d_{z}}$. 
In the LQG problem of the conventional POSC, the memory dimension $d_{z}$ is the same with the state dimension $d_{x}$. 
The controller of ML-POSC determines the control $u_{t}$ based on the memory $z_{t}$, i.e., $u_{t}=u(t,z_{t})$. 
The memory $z_{t}$ is assumed to evolve by the following SDE: 
\begin{align}
	dz_{t}=v_{t}dt+\kappa_{t}dy_{t},\label{eq: memory SDE LQG-POSC}
\end{align}
where $z_{0}=\mu_{0,xx}$, and $v_{t}=v(t,z_{t})\in\mb{R}^{d_{z}}$ and $\kappa_{t}=\kappa(t,z_{t})\in\mb{R}^{d_{z}\times d_{y}}$ are the memory controls. 
We note that the LQG problem of the conventional POSC does not consider the memory noise. 
The objective function of ML-POSC is given by the following expected cumulative cost function: 
\begin{align}	
	J[u,v,\kappa]:=\mb{E}_{p(x_{0:T},y_{0:T},z_{0:T};u,v,\kappa)}\left[\int_{0}^{T}\left(x_{t}^{T}Q(t)x_{t}+u_{t}^{T}R(t)u_{t}\right)dt+x_{T}^{T}Px_{T}\right]. 
	\label{eq: OF of ML-POSC LQG-POSC}
\end{align}
We note that the LQG problem of the conventional POSC does not consider the memory control cost. 
ML-POSC optimizes $u$, $v$, and $\kappa$ based on $J[u,v,\kappa]$ as follows: 
\begin{align}	
	u^{*},v^{*},\kappa^{*}:=\argmin_{u,v,\kappa}J[u,v,\kappa].
	\label{eq: ML-POSC LQG-POSC}
\end{align}

In the LQG problem of the conventional POSC, the probability of the extended state $s_{t}$ (\ref{eq: extended state}) is given by the Gaussian distribution $p_{t}(s_{t})=\mcal{N}(s_{t}|\mu(t),\Sigma(t))$. 
The posterior probability of the state $x_{t}$ given the memory $z_{t}$ is also given by the Gaussian distribution $p_{t}(x_{t}|z_{t})=\mcal{N}(x_{t}|\mu_{x|z}(t,z_{t}),\Sigma_{x|z}(t))$, where $\mu_{x|z}(t,z_{t})$ and $\Sigma_{x|z}(t)$ are given as follows: 
\begin{align}
	\mu_{x|z}(t,z_{t})&=\mu_{x}(t)+\Sigma_{xz}(t)\Sigma_{zz}^{-1}(t)(z_{t}-\mu_{z}(t)),\label{eq: conditional expectation LQG-POSC}\\
	\Sigma_{x|z}(t)&=\Sigma_{xx}(t)-\Sigma_{xz}(t)\Sigma_{zz}^{-1}(t)\Sigma_{zx}(t).\label{eq: conditional covariance LQG-POSC}
\end{align}

\begin{theo}\label{theo: optimal control of ML-POSC in LQG-POSC}
In LQG problem of the conventional POSC, the optimal control functions are given by
\begin{align}
	u^{*}(t,z)&=-R^{-1}B^{T}\Psi z,\\
	v^{*}(t,z)&=\left(A-BR^{-1}B^{T}\Psi-\Sigma_{x|z}H^{T}(\gamma\gamma^{T})^{-1}H\right)z,\\
	\kappa^{*}(t,z)&=\Sigma_{x|z}H^{T}(\gamma\gamma^{T})^{-1},\label{optimal control of ML-POSC in LQG-POSC kappa}
\end{align}
where $\mu_{x|z}(t,z_{t})=z_{t}$ and $\Sigma_{x|z}(t)$ are the solutions of the following equations: 
\begin{align}
	&dz_{t}=\left(A-BR^{-1}B^{T}\Psi\right)z_{t}dt+\Sigma_{x|z}H^{T}(\gamma\gamma^{T})^{-1}\left(dy_{t}-Hz_{t}dt\right),\label{eq: Kalman filter mu in ML-POSC}\\
	&\frac{d\Sigma_{x|z}}{dt}=\sigma\sigma^{T}+A\Sigma_{x|z}+\Sigma_{x|z}A^{T}-\Sigma_{x|z}H^{T}(\gamma\gamma^{T})^{-1}H\Sigma_{x|z},\label{eq: Kalman filter Sigma in ML-POSC}
\end{align}
where $z_{0}=\mu_{x,0}$ and $\Sigma_{x|z}(0)=\Sigma_{xx,0}$. 
$\Psi(t)$ is the solution of the Riccati equation (\ref{eq: ODE of Psi LQG-POSC}).  
\end{theo}

\begin{proof}
The proof is shown in Appendix D.  
\end{proof}

In the LQG problem of the conventional POSC, the optimal memory dynamics of ML-POSC (\ref{eq: Kalman filter mu in ML-POSC}), (\ref{eq: Kalman filter Sigma in ML-POSC}) corresponds to the Kalman filter (\ref{eq: Kalman filter mu}), (\ref{eq: Kalman filter Sigma}). 
Furthermore, ML-POSC reproduces the Riccati equation (\ref{eq: ODE of Psi LQG-POSC}). 

\section{LINEAR-QUADRATIC-GAUSSIAN PROBLEM WITH MEMORY LIMITATION}\label{sec: LQG problem with memory limitation}
The LQG problem of the conventional POSC does not consider memory limitation. 
Especially, it does not consider the memory noise and cost. 
Furthermore, because the memory dimension is restricted to the state dimension in the finite-dimensional Kalman filter, the memory dimension cannot be determined according to a given controller. 
ML-POSC can generalize the LQG problem to include the memory limitation. 
In this section, we discuss the LQG problem with memory limitation based on ML-POSC. 

\subsection{Problem formulation}\label{sec: LQG Problem formulation}
In this subsection, we formulate the LQG problem with memory limitation. 
The state and observation SDEs are the same as in the previous section, which are given by (\ref{eq: state SDE LQG-POSC}) and (\ref{eq: observation SDE LQG-POSC}), respectively. 
The controller of ML-POSC determines the control $u_{t}\in\mb{R}^{d_{u}}$ based on the memory $z_{t}\in\mb{R}^{d_{z}}$, i.e., $u_{t}=u(t,z_{t})$. 
Unlike the LQG problem of the conventional POSC, the memory dimension $d_{z}$ is not necessarily the same with the state dimension $d_{x}$. 

The memory $z_{t}$ is assumed to evolve by the following SDE: 
\begin{align}
	dz_{t}&=v_{t}dt+\kappa(t)dy_{t}+\eta(t)d\xi_{t},\label{eq: memory SDE LQG-ML-POSC}
\end{align}
where $z_{0}$ obeys the Gaussian distribution $p_{0}(z_{0})=\mcal{N}\left(z_{0}\left|\mu_{z,0},\Sigma_{zz,0}\right.\right)$, $\xi_{t}\in\mb{R}^{d_{\xi}}$ is the standard Wiener process, and $v_{t}=v(t,z_{t})\in\mb{R}^{d_{v}}$ is the control. 
Because the initial condition $z_{0}$ is stochastic and the memory SDE (\ref{eq: memory SDE LQG-ML-POSC}) includes the intrinsic stochasticity $d\xi_{t}$, the LQG problem of ML-POSC can consider the memory noise explicitly. 
We note that $\kappa(t)$ is independent of the memory $z_{t}$. 
If $\kappa(t)$ depends on the memory $z_{t}$, the memory SDE (\ref{eq: memory SDE LQG-ML-POSC}) becomes non-linear and non-Gaussian. 
As a result, the optimal control functions cannot be derived explicitly in this case. 
In order to keep the memory SDE (\ref{eq: memory SDE LQG-ML-POSC}) linear and Gaussian obtaining the optimal control functions explicitly, we restrict $\kappa(t)$ being independent of the memory $z_{t}$ in the LQG problem with memory limitation. 
The LQG problem without memory limitation is the special case where the optimal control $\kappa_{t}^{*}=\kappa^{*}(t,z_{t})$ in (\ref{optimal control of ML-POSC in LQG-POSC kappa}) does not depend on the memory $z_{t}$.  

The objective function is given by the following expected cumulative cost function: 
\begin{align}	
	J[u,v]:=\mb{E}_{p(x_{0:T},y_{0:T},z_{0:T};u,v)}\left[\int_{0}^{T}\left(x_{t}^{T}Q(t)x_{t}+u_{t}^{T}R(t)u_{t}+v_{t}^{T}M(t)v_{t}\right)dt+x_{T}^{T}Px_{T}\right],
	\label{eq: OF of ML-POSC LQG-ML-POSC}
\end{align}
where $Q(t)\succeq O$, $R(t)\succ O$, $M(t)\succ O$, and $P\succeq O$. 
Because the cost function includes $v_{t}^{T}M(t)v_{t}$, the LQG problem of ML-POSC can consider the memory control cost explicitly. 
ML-POSC optimizes the state control function $u$ and the memory control function $v$ based on the objective function $J[u,v]$ as follows: 
\begin{align}	
	u^{*},v^{*}:=\argmin_{u,v}J[u,v].
	\label{eq: ML-POSC LQG-ML-POSC}
\end{align}
For the sake of simplicity, we do not optimize $\kappa(t)$ although we can in the similar way.

\subsection{Problem reformulation}\label{sec: LQG Problem reformulation}
Although the formulation in the previous subsection clarifies its relationship with that in the previous section,  it is inconvenient for further mathematical investigations. 
In order to resolve this problem, we reformulate the LQG problem with memory limitation based on the extended state $s_{t}$ (\ref{eq: extended state}). 
The formulation in this subsection is simpler and more general than that in the previous subsection. 

In the LQG problem with memory limitation, the extended state SDE (\ref{eq: extended state SDE}) is given as follows:  
\begin{align}
	ds_{t}=\left(\tilde{A}(t)s_{t}+\tilde{B}(t)\tilde{u}_{t}\right)dt+\tilde{\sigma}(t)d\tilde{\omega}_{t},\label{SDE of LQG}
\end{align}
where $s_{0}$ obeys the Gaussian distribution $p_{0}(s_{0}):=\mcal{N}\left(s_{0}\left|\mu_{0},\Sigma_{0}\right.\right)$, $\tilde{\omega}_{t}\in\mb{R}^{d_{\tilde{\omega}}}$ is the standard Wiener process, and $\tilde{u}_{t}=\tilde{u}(t,z_{t})\in\mb{R}^{d_{\tilde{u}}}$ is the control.
The extended state SDE (\ref{SDE of LQG}) includes the previous state, observation, and memory SDEs (\ref{eq: state SDE LQG-POSC}), (\ref{eq: observation SDE LQG-POSC}), (\ref{eq: memory SDE LQG-ML-POSC}) as a special case because they can be represented as follows: 
\begin{align}
	ds_{t}
	&=\left(\left(\begin{array}{cc}
		A&O\\
		\kappa H&O\\
	\end{array}\right)s_{t}
	+\left(\begin{array}{cc}
		B&O\\
		O&I\\
	\end{array}\right)
	\left(\begin{array}{c}
		u_{t}\\ v_{t}\\
	\end{array}\right)\right)dt
	+\left(\begin{array}{ccc}
		\sigma &O&O\\
		O&\kappa\gamma&\eta\\
	\end{array}\right)
	\left(\begin{array}{c}
		d\omega_{t}\\
		d\nu_{t}\\
		d\xi_{t}\\
	\end{array}\right),
\end{align}
where $p_{0}(s_{0})=p_{0}(x_{0})p_{0}(z_{0})$.

The objective function (\ref{eq: OF of GML-POSC}) is given the following expected cumulative cost function: 
\begin{align}
	J[\tilde{u}]:=\mb{E}_{p(s_{0:T};\tilde{u})}\left[\int_{0}^{T}\left(s_{t}^{T}\tilde{Q}(t)s_{t}+\tilde{u}_{t}^{T}\tilde{R}(t)\tilde{u}_{t}\right)dt+s_{T}^{T}\tilde{P}s_{T}\right],\label{OF of LQG}
\end{align}
where $\tilde{Q}(t)\succeq O$, $\tilde{R}(t)\succ O$, and $\tilde{P}\succeq O$. 
This objective function (\ref{OF of LQG}) includes the previous objective function (\ref{eq: OF of ML-POSC LQG-ML-POSC}) as a special case because it can be represented as follows: 
\begin{align}
	J[u,v]:=\mb{E}_{p(s_{0:T};u,v)}\left[\int_{0}^{T}\left(s_{t}^{T}\left(\begin{array}{cc}
		Q&O\\
		O&O\\
	\end{array}\right)s_{t}
	+\left(\begin{array}{c}
		u_{t}\\ v_{t}\\
	\end{array}\right)^{T}
	\left(\begin{array}{cc}
		R&O\\
		O&M\\
	\end{array}\right)
	\left(\begin{array}{c}
		u_{t}\\ v_{t}\\
	\end{array}\right)\right)dt+s_{T}^{T}
	\left(\begin{array}{cc}
		P&O\\
		O&O\\
	\end{array}\right)s_{T}\right].
\end{align}

The objective of the LQG problem with memory limitation is to find the optimal control function $\tilde{u}^{*}$ that minimizes the objective function $J[\tilde{u}]$ as follows: 
\begin{align}	
	\tilde{u}^{*}:=\argmin_{\tilde{u}}J[\tilde{u}].
	\label{eq: LQG}
\end{align}

In the following subsection, we mainly consider the formulation of this subsection rather than that of the previous subsection because it is simpler and more general. 
Moreover, we omit $\tilde{\cdot}$ for the notational simplicity. 

\subsection{Derivation of optimal control function}\label{sec: LQG Derivation of optimal control function}
In this subsection, we derive the optimal control function of the LQG problem with memory limitation by applying Theorem \ref{theo: optimal control of GML-POSC}.  
In the LQG problem with memory limitation, the probability of the extended state $s$ at time $t$ is given by the Gaussian distribution $p_{t}(s):=\mcal{N}\left(s|\mu(t),\Sigma(t)\right)$. 
By defining the stochastic extended state $\hat{s}:=s-\mu$, 
$\mb{E}_{p_{t}(x|z)}\left[s\right]$ is given as follows: 
\begin{align}
	\mb{E}_{p_{t}(x|z)}\left[s\right]=K(t)\hat{s}+\mu(t), 
	\label{eq: conditional mean vector of LQG}
\end{align}
where $K(t)$ is defined by 
\begin{align}
	K(t):=\left(\begin{array}{cc}
		O&\Sigma_{xz}(t)\Sigma_{zz}^{-1}(t)\\ 
		O&I\\
	\end{array}\right).
	\label{eq: inference gain}
\end{align}
By applying Theorem \ref{theo: optimal control of GML-POSC} to the LQG problem with memory limitation, we obtain the following theorem: 

\begin{theo}\label{theo: optimal control of LQG in ML-POSC}
In the LQG problem with memory limitation, the optimal control function is given by
\begin{align}
	u^{*}(t,z)=-R^{-1}B^{T}\left(\Pi K\hat{s}+\Psi\mu\right),
	\label{eq: optimal control of LQG}
\end{align}
where $K(t)$ (\ref{eq: inference gain}) depends on $\Sigma(t)$, and $\mu(t)$ and $\Sigma(t)$ are the solutions of the following ordinary differential equations: 
\begin{align}
	\frac{d\mu}{dt}&=\left(A-BR^{-1}B^{T}\Psi\right)\mu\label{eq: ODE of mu},\\
	\frac{d\Sigma}{dt}&=\sigma\sigma^{T}+\left(A-BR^{-1}B^{T}\Pi K\right)\Sigma+\Sigma\left(A-BR^{-1}B^{T}\Pi K\right)^{T},\label{eq: ODE of Sigma}
\end{align}
where $\mu(0)=\mu_{0}$ and $\Sigma(0)=\Sigma_{0}$. 
$\Psi(t)$ and $\Pi(t)$ are the solutions of the following ordinary differential equations: 
\begin{align}
	-\frac{d\Psi}{dt}&=Q+A^{T}\Psi+\Psi A -\Psi BR^{-1}B^{T}\Psi,\label{eq: ODE of Psi}\\
	-\frac{d\Pi}{dt}&=Q+A^{T}\Pi+\Pi A-\Pi BR^{-1}B^{T}\Pi+(I-K)^{T}\Pi BR^{-1}B^{T}\Pi (I-K),\label{eq: ODE of Pi}
\end{align}
where $\Psi(T)=\Pi(T)=P$. 
\end{theo}

\begin{proof}
The proof is shown in Appendix E.
\end{proof}

(\ref{eq: ODE of Psi}) is the Riccati equation \cite{bensoussan_stochastic_1992,yong_stochastic_1999,bensoussan_estimation_2018}, which also appears in the LQG problem without memory limitation (\ref{eq: ODE of Psi LQG-POSC}). 
In contrast, (\ref{eq: ODE of Pi}) is a new equation of the LQG problem with memory limitation, which is called the partially observable Riccati equation in this paper. 
Because estimation and control are not clearly separated in the LQG problem with memory limitation, the Riccati equation (\ref{eq: ODE of Psi}) for control is modified to include estimation, which corresponds to the partially observable Riccati equation (\ref{eq: ODE of Pi}). 
As a result, the partially observable Riccati equation (\ref{eq: ODE of Pi}) may improve estimation as well as control. 

In order to support this interpretation, we analyze the partially observable Riccati equation (\ref{eq: ODE of Pi}) by comparing it with the Riccati equation (\ref{eq: ODE of Psi}). 
Since only the last term of (\ref{eq: ODE of Pi}) is different from (\ref{eq: ODE of Psi}), we denote it as follows: 
\begin{align}
	\mcal{Q}:=(I-K)^{T}\Pi BR^{-1}B^{T}\Pi (I-K). 
\end{align}
$\mcal{Q}$ can be calculated as follows: 
\begin{align}
	\mcal{Q}=\left(\begin{array}{cc}
		\mcal{P}_{xx}&-\mcal{P}_{xx}\Sigma_{xz}\Sigma_{zz}^{-1}\\ 
		-\Sigma_{zz}^{-1}\Sigma_{zx}\mcal{P}_{xx}&\Sigma_{zz}^{-1}\Sigma_{zx}\mcal{P}_{xx}\Sigma_{xz}\Sigma_{zz}^{-1}\\
	\end{array}\right),
	\label{eq: last term of POR eq}
\end{align}
where $\mcal{P}_{xx}:=(\Pi BR^{-1}B^{T}\Pi)_{xx}$. 
Because $\mcal{P}_{xx}\succeq O$ and $\Sigma_{zz}^{-1}\Sigma_{zx}\mcal{P}_{xx}\Sigma_{xz}\Sigma_{zz}^{-1}\succeq O$, 
$\Pi_{xx}$ and $\Pi_{zz}$ may be larger than $\Psi_{xx}$ and $\Psi_{zz}$, respectively. 
Because $\Pi_{xx}$ and $\Pi_{zz}$ are the negative feedback gains of the state $x$ and the memory $z$, respectively, 
$\mcal{Q}$ may decrease $\Sigma_{xx}$ and $\Sigma_{zz}$. 
Moreover, when $\Sigma_{xz}$ is positive/negative, $\Pi_{xz}$ may be smaller/larger than $\Psi_{xz}$, which may increase/decrease $\Sigma_{xz}$. 
The similar discussion is possible for $\Sigma_{zx}$, $\Pi_{zx}$, and $\Psi_{zx}$ because $\Sigma$, $\Pi$, and $\Psi$ are symmetric matrices. 
As a result, $\mcal{Q}$ may decrease the following conditional covariance matrix: 
\begin{align}
	\Sigma_{x|z}:=\Sigma_{xx}-\Sigma_{xz}\Sigma_{zz}^{-1}\Sigma_{zx}, \label{eq: estimation error}
\end{align}
which corresponds to the estimation error of the state from the memory. 
Therefore, the partially observable Riccati equation (\ref{eq: ODE of Pi}) may improve estimation as well as control, which is different from the Riccati equation (\ref{eq: ODE of Psi}). 

Because the problem of Sec. \ref{sec: LQG Problem formulation} is more special than that of Sec. \ref{sec: LQG Problem reformulation}, we can make a more specific discussion. 
In the problem of Sec. \ref{sec: LQG Problem formulation}, $\Psi_{xx}$ is the same with the solution of the Riccati equation of the conventional POSC (\ref{eq: ODE of Psi LQG-POSC}), and $\Psi_{xz}=O$, $\Psi_{zx}=O$, and $\Psi_{zz}=O$ are satisfied. 
As a result, the memory control does not appear in the Riccati equation of ML-POSC (\ref{eq: ODE of Psi}). 
In contrast, because of the last term of the partially observable Riccati equation (\ref{eq: ODE of Pi}), $\Pi_{xx}$ is not the solution of the Riccati equation (\ref{eq: ODE of Psi LQG-POSC}), and $\Pi_{xz}\neq O$, $\Pi_{zx}\neq O$, and $\Pi_{zz}\neq O$ are satisfied. 
As a result, the memory control appears in the partially observable Riccati equation (\ref{eq: ODE of Pi}), which may improve the state estimation.

\subsection{Comparison with completely observable stochastic control}
In this subsection, we compare ML-POSC with the COSC of the extended state. 
By applying Theorem \ref{theo: optimal control of COSC} into the LQG problem, the optimal control function of the COSC of the extended state can be obtained as follows: 
\begin{theo}\label{theo: optimal control of LQG in COSC}
In the LQG problem, the optimal control function of the COSC of the extended state is given by
\begin{align}
	u^{*}(t,s)=-R^{-1}B^{T}\Psi s
	=-R^{-1}B^{T}\left(\Psi\hat{s}+\Psi\mu\right),
	\label{eq: optimal control of LQG COSC}
\end{align}
where $\Psi(t)$ is the solution of the Riccati equation (\ref{eq: ODE of Psi}). 
\end{theo}

\begin{proof}
The proof is shown in \cite{yong_stochastic_1999,bensoussan_estimation_2018}. 
\end{proof}

The optimal control function of the COSC of the extended state (\ref{eq: optimal control of LQG COSC}) can be derived intuitively from that of ML-POSC (\ref{eq: optimal control of LQG}). 
In ML-POSC, $K\hat{s}=\mb{E}_{p_{t}(x|z)}\left[\hat{s}\right]$ is the estimator of the stochastic extended state. 
In the COSC of the extended state, because the stochastic extended state is completely observable, its estimator is given by $\hat{s}$, which corresponds to $K=I$. 
By changing the definition of $K$ from (\ref{eq: inference gain}) to $K=I$, the partially observable Riccati equation (\ref{eq: ODE of Pi}) is reduced to the Riccati equation (\ref{eq: ODE of Psi}), 
and the optimal control function of ML-POSC (\ref{eq: optimal control of LQG}) is reduced to that of COSC (\ref{eq: optimal control of LQG COSC}). 
As a result, the optimal control function of ML-POSC (\ref{eq: optimal control of LQG}) can be interpreted as the generalization of that of COSC (\ref{eq: optimal control of LQG COSC}). 

While the second term is the same between (\ref{eq: optimal control of LQG}) and (\ref{eq: optimal control of LQG COSC}), the first term is different. 
The second term is the control of the expected extended state $\mu$, which does not depend on the realization. 
In contrast, the first term is the control of the stochastic extended state $\hat{s}$, which depends on the realization. 
The first term has two different points: 
(i) The estimators of the stochastic extended state in COSC and ML-POSC are given by $\hat{s}$ and $K\hat{s}=\mb{E}_{p_{t}(x|z)}\left[\hat{s}\right]$, respectively. 
This is reasonable because ML-POSC needs to estimate the state from the memory. 
(ii) The control gains of the stochastic extended state in COSC and ML-POSC are given by $\Psi$ and $\Pi$, respectively. 
While $\Psi$ improves only control, $\Pi$ improves estimation as well as control. 

\subsection{Numerical algorithm}
In the LQG problem, the partial differential equations are reduced to the ordinary differential equations. 
The FP equation (\ref{eq: FP eq}) is reduced to (\ref{eq: ODE of mu}) and (\ref{eq: ODE of Sigma}), and the HJB equation (\ref{eq: HJB eq}) is reduced to (\ref{eq: ODE of Psi}) and (\ref{eq: ODE of Pi}). 
As a result, the optimal control function (\ref{eq: optimal control of LQG}) can be obtained more easily in the LQG problem. 

The Riccati equation (\ref{eq: ODE of Psi}) can be solved backward from the terminal condition. 
In contrast, the partially observable Riccati equation (\ref{eq: ODE of Pi}) cannot be solved in the same way as the Riccati equation (\ref{eq: ODE of Psi}) because it depends on the forward equation of $\Sigma$ (\ref{eq: ODE of Sigma}) through $K$ (\ref{eq: inference gain}). 
Because the forward equation of $\Sigma$ (\ref{eq: ODE of Sigma}) also depends on the backward equation of $\Pi$ (\ref{eq: ODE of Pi}), they must be solved simultaneously. 

The similar problem also appears in the mean-field game and control, and numerous numerical methods have been developed \cite{lauriere_numerical_2021}. 
In this paper, we solve the system of (\ref{eq: ODE of Sigma}) and (\ref{eq: ODE of Pi}) by using the forward-backward sweep method (fixed-point iteration method), which computes (\ref{eq: ODE of Sigma}) and (\ref{eq: ODE of Pi}) alternately \cite{lauriere_numerical_2021,tottori_notitle_2022}. 
In ML-POSC, the convergence of the forward-backward sweep method is guaranteed \cite{tottori_notitle_2022}. 

\section{NUMERICAL EXPERIMENTS}\label{sec: Numerical experiments}
In this section, we demonstrate the effectiveness of ML-POSC by the numerical experiments on the LQG problem with memory limitation and the non-LQG problem. 

\subsection{LQG problem with memory limitation}
\begin{figure*}[t]
\begin{center}
	\begin{minipage}[t][][b]{45mm}
	(a)
	\end{minipage}
	\begin{minipage}[t][][b]{45mm}
	(b)
	\end{minipage}
	\begin{minipage}[t][][b]{45mm}
	(c)
	\end{minipage}\\
	\begin{minipage}[t][][b]{45mm}
		\includegraphics[width=45mm]{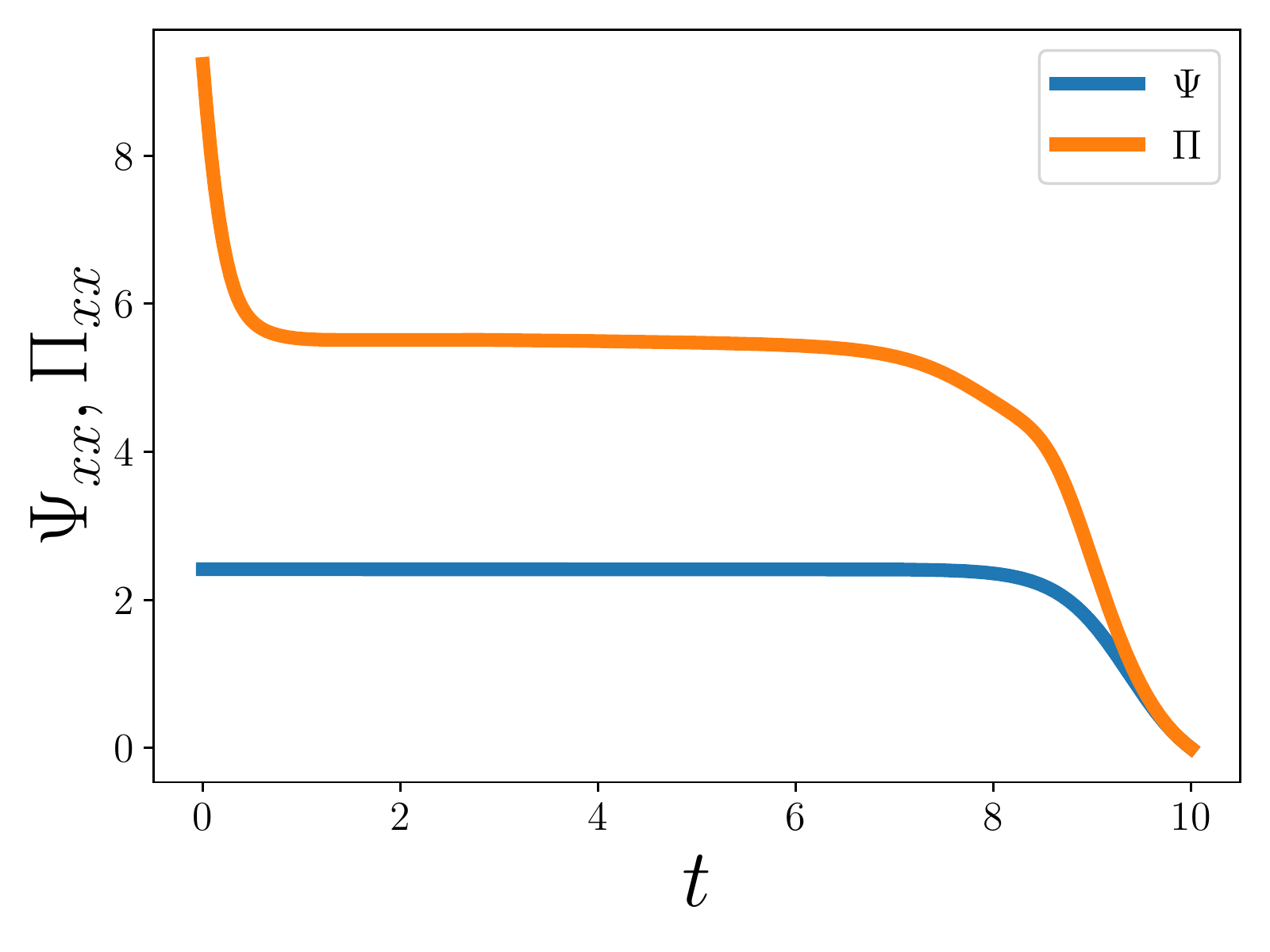}
	\end{minipage}
	\begin{minipage}[t][][b]{45mm}
		\includegraphics[width=45mm]{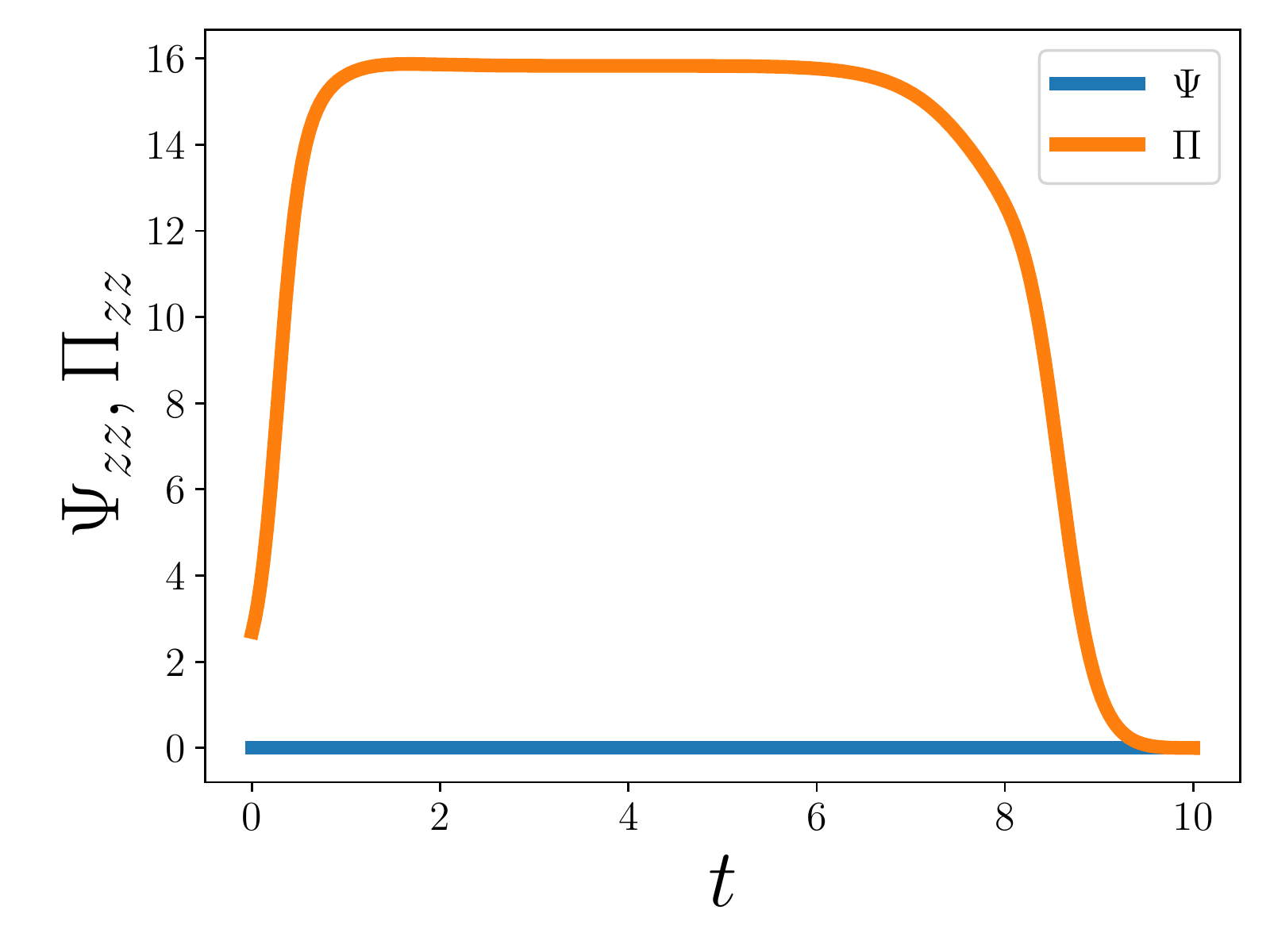}
	\end{minipage}
	\begin{minipage}[t][][b]{45mm}
		\includegraphics[width=45mm]{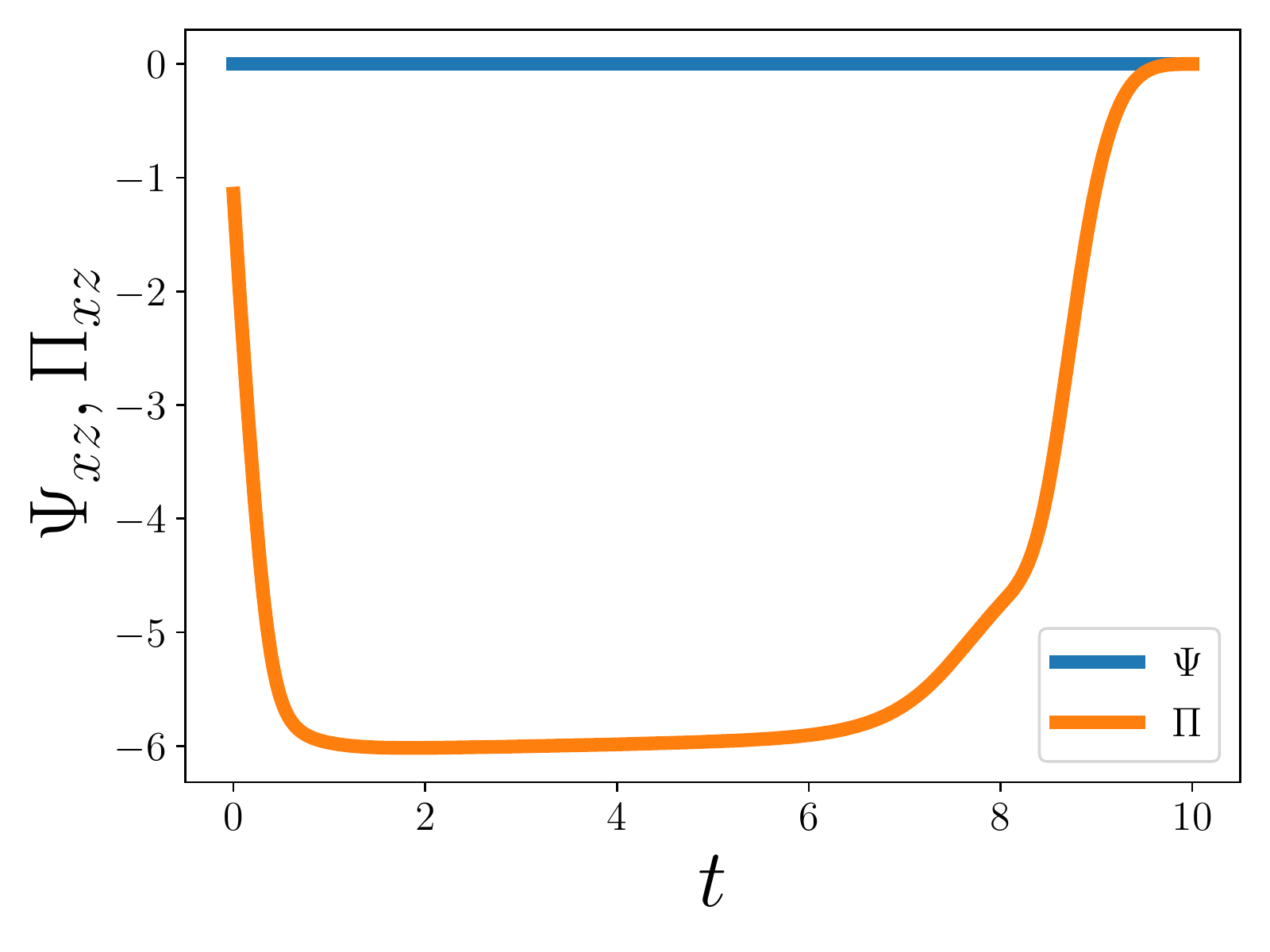}
	\end{minipage}\\
	\begin{minipage}[t][][b]{45mm}
	(d)
	\end{minipage}
	\begin{minipage}[t][][b]{45mm}
	(e)
	\end{minipage}
	\begin{minipage}[t][][b]{45mm}
	(f)
	\end{minipage}\\
	\begin{minipage}[t][][b]{45mm}
		\includegraphics[width=45mm]{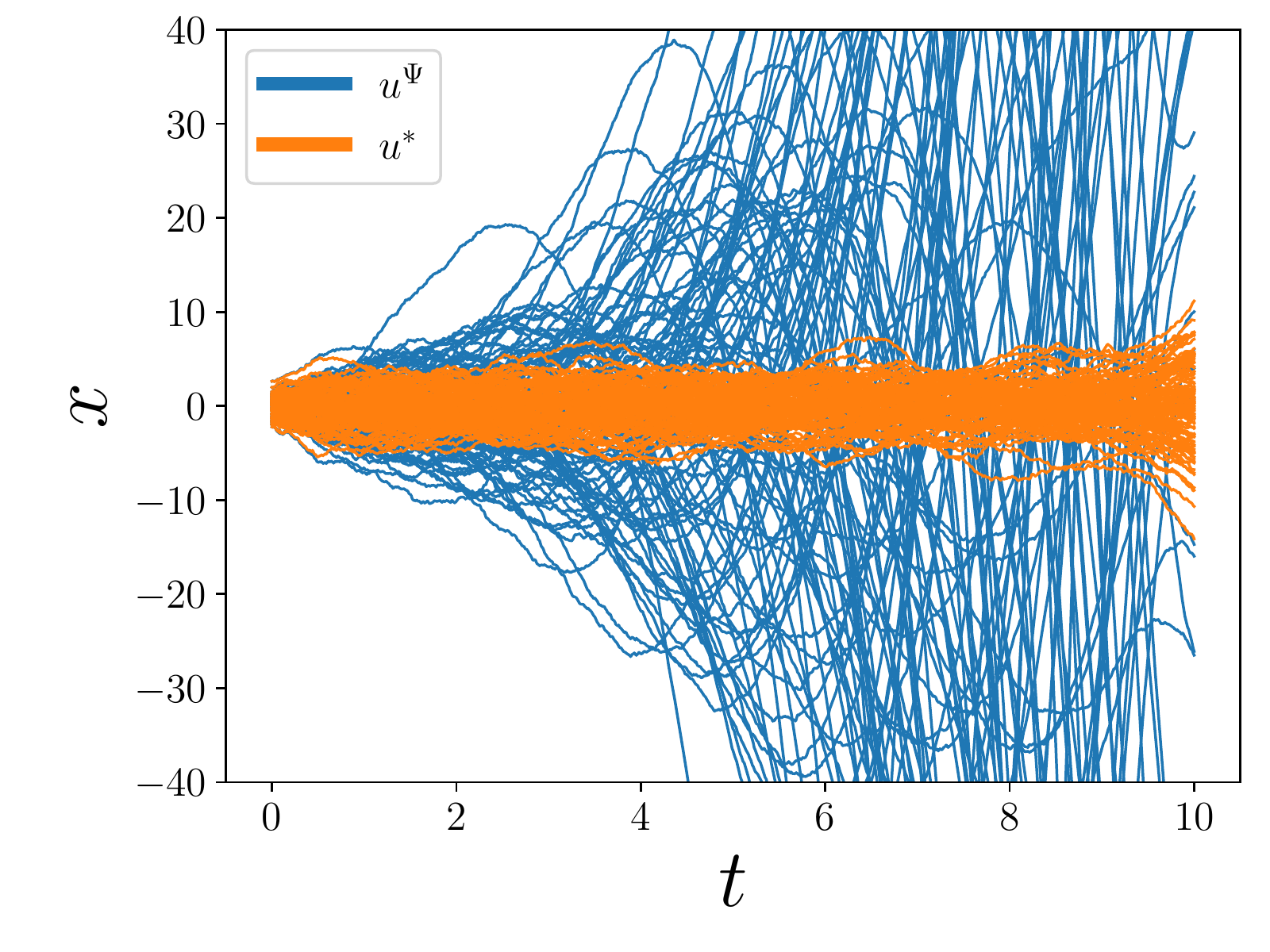}
	\end{minipage}
	\begin{minipage}[t][][b]{45mm}
		\includegraphics[width=45mm]{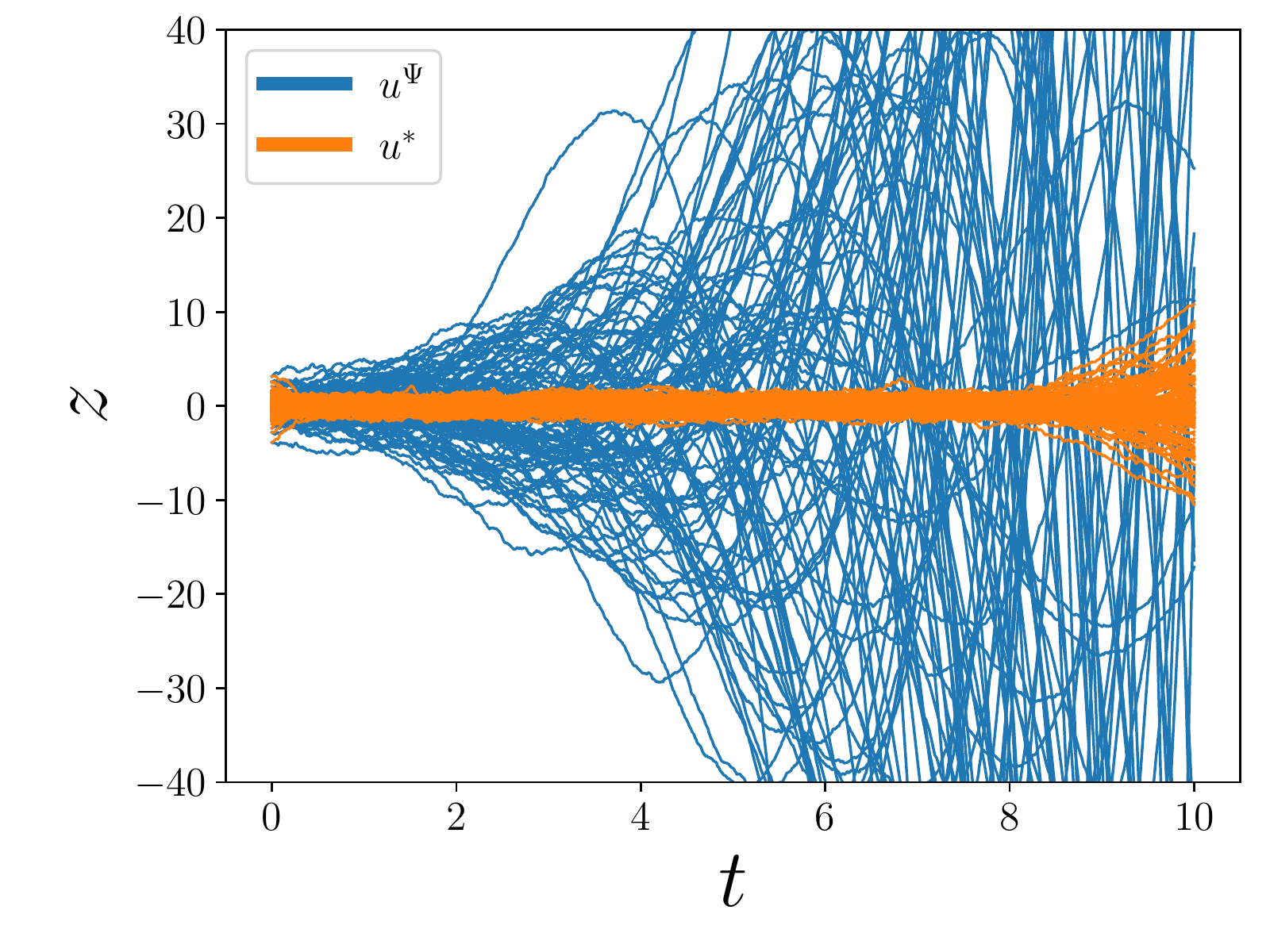}
	\end{minipage}
	\begin{minipage}[t][][b]{45mm}
		\includegraphics[width=45mm]{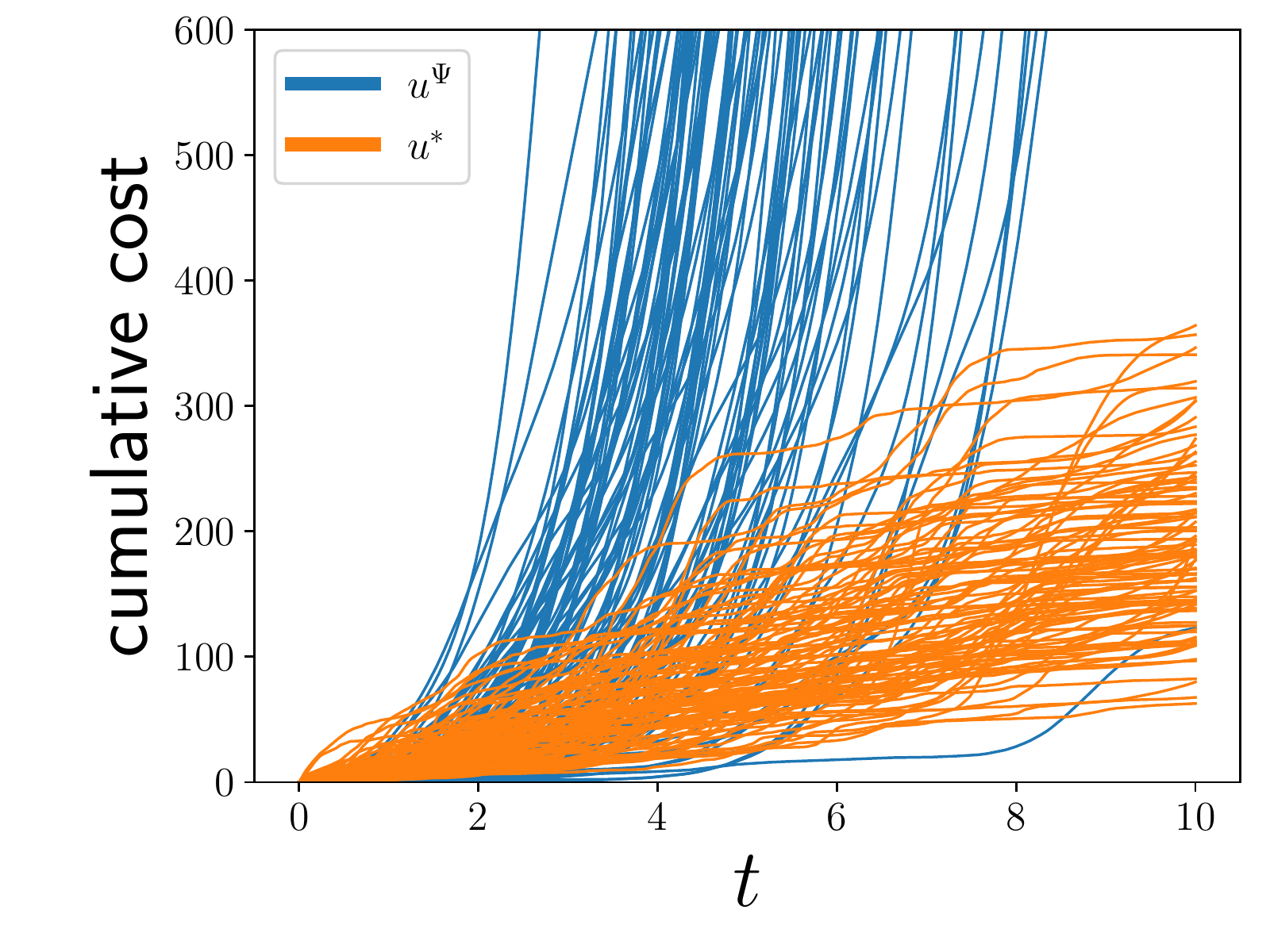}
	\end{minipage}
	\caption{
	Numerical simulation of the LQG problem with memory limitation. 
	(a,b,c) Trajectories of $\Psi(t)$ and $\Pi(t)$. 
	(d,e,f) Stochastic behaviors of the state $x_{t}$ (d), the memory $z_{t}$ (e), and the cumulative cost (f) for 100 samples. 
	The expectation of the cumulative cost at $t=10$ corresponds to the objective function (\ref{eq: OF of ML-POSC LQG NE}).
	Blue and orange curves are controlled by (\ref{eq: Psi control of LQG}) and (\ref{eq: optimal control of LQG}), respectively. 
	}
	\label{fig: LQG}
\end{center}
\end{figure*}

In this subsection, we show the significance of the partially observable Riccati equation (\ref{eq: ODE of Pi}) by a numerical experiment of the LQG problem with memory limitation. 
We consider the state $x_{t}\in\mb{R}$, the observation $y_{t}\in\mb{R}$, and the memory $z_{t}\in\mb{R}$, which evolve by the following SDEs: 
\begin{align}
	dx_{t}&=\left(x_{t}+u_{t}\right)dt+d\omega_{t},\label{eq: state SDE LQG NE}\\
	dy_{t}&=x_{t}dt+d\nu_{t},\label{eq: observation SDE LQG NE}\\
	dz_{t}&=v_{t}dt+dy_{t},\label{eq: memory SDE LQG NE}
\end{align}
where $x_{0}$ and $z_{0}$ obey the standard Gaussian distributions, $y_{0}$ is an arbitrary real number, $\omega_{t}\in\mb{R}$ and $\nu_{t}\in\mb{R}$ are independent standard Wiener processes, $u_{t}=u(t,z_{t})\in\mb{R}$ and $v_{t}=v(t,z_{t})\in\mb{R}$ are the controls. 
The objective function to be minimized is given as follows: 
\begin{align}	
	J[u,v]:=\mb{E}\left[\int_{0}^{10}\left(x_{t}^{2}+u_{t}^{2}+v_{t}^{2}\right)dt\right].
	\label{eq: OF of ML-POSC LQG NE}
\end{align} 
Therefore, the objective of this problem is to minimize the state variance by the small state control and memory control. 
Because this problem includes the memory control cost, it corresponds to the LQG problem with memory limitation. 

Fig. \ref{fig: LQG} (a,b,c) shows the trajectories of $\Psi(t)$ and $\Pi(t)$. 
$\Pi_{xx}$ and $\Pi_{zz}$ are larger than $\Psi_{xx}$ and $\Psi_{zz}$, respectively, and $\Pi_{xz}$ is smaller than $\Psi_{xz}$, 
which is consistent with our discussion in Sec. \ref{sec: LQG Derivation of optimal control function}. 
Therefore, the partially observable Riccati equation may reduce the estimation error of the state from the memory. 
Moreover, while the memory control does not appear in the Riccati equation ($\Psi_{xz}=\Psi_{zz}=0$), it appears in the partially observable Riccati equation ($\Pi_{xz}\neq 0$, $\Pi_{zz}\neq 0$), which is also consistent with our discussion in Sec. \ref{sec: LQG Derivation of optimal control function}. 
As a result, the memory control plays an important role in estimating the state from the memory. 

In order to clarify the significance of the partially observable Riccati equation (\ref{eq: ODE of Pi}), 
we compare the performance of the optimal control function (\ref{eq: optimal control of LQG}) with that of the following control function: 
\begin{align}
	u^{\Psi}(t,z)=-R^{-1}B^{T}\left(\Psi K\hat{s}+\Psi\mu\right), 
	\label{eq: Psi control of LQG}
\end{align}
which replaces $\Pi$ with $\Psi$. 
This result is shown in Fig. \ref{fig: LQG} (d,e,f).  
In the control function (\ref{eq: Psi control of LQG}), the distributions of the state and the memory are unstable, and the cumulative cost diverges. 
By contrast, in the optimal control function (\ref{eq: optimal control of LQG}), the distributions of the state and the memory are stable, and the cumulative cost is smaller. 
This result indicates that the partially observable Riccati equation (\ref{eq: ODE of Pi}) plays an important role in the LQG problem with memory limitation. 

\subsection{Non-LQG problem}
\begin{figure}[t]
\begin{center}
	\begin{minipage}[t][][b]{42mm}
	(a)
	\end{minipage}
	\begin{minipage}[t][][b]{42mm}
	(b)
	\end{minipage}\\
	\begin{minipage}[t][][b]{42mm}
		\includegraphics[width=42mm]{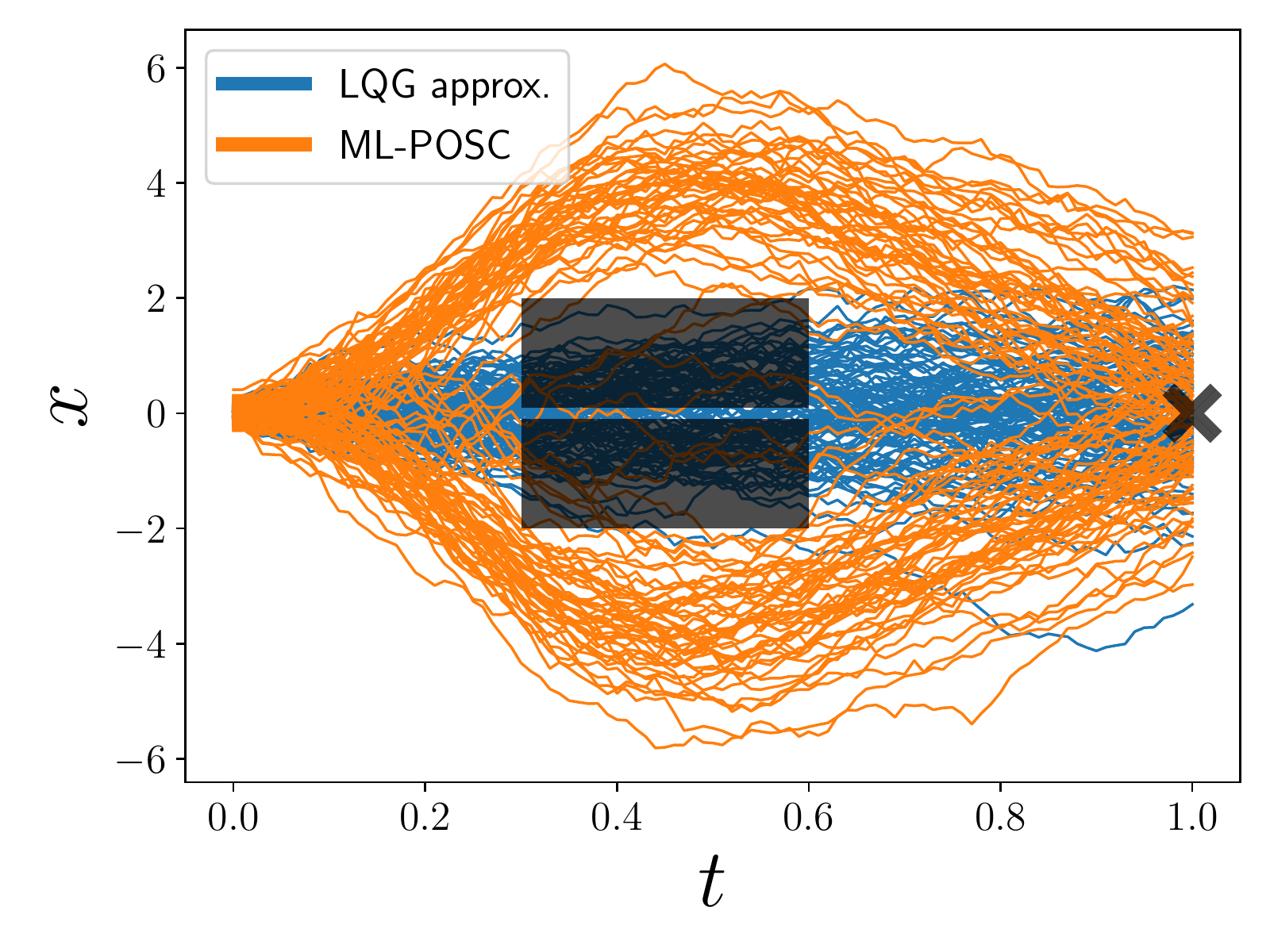}
	\end{minipage}
	\begin{minipage}[t][][b]{42mm}
		\includegraphics[width=42mm]{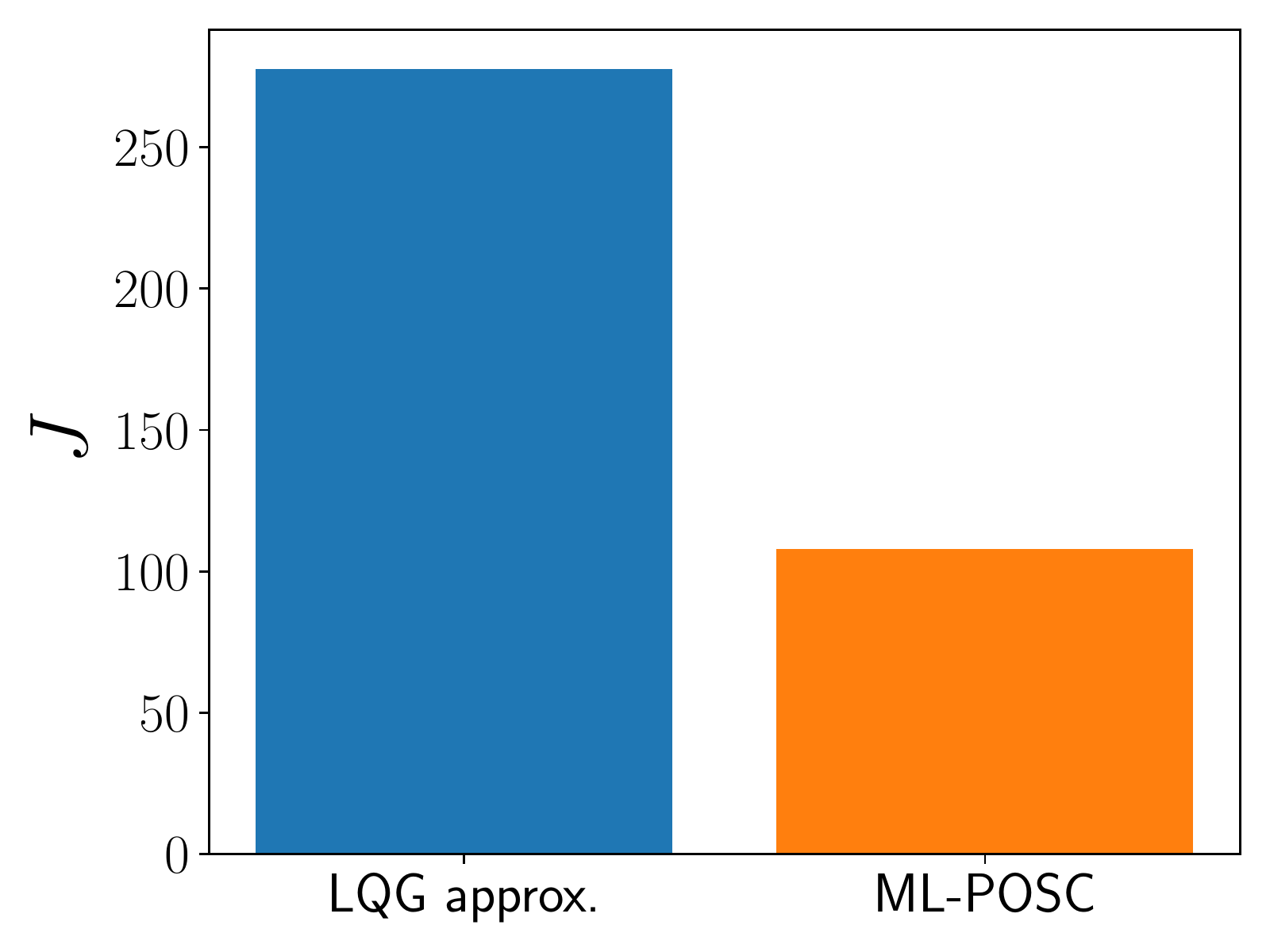}
	\end{minipage}
	\caption{
	Numerical simulation of the non-LQG problem for the local LQG approximation (blue) and ML-POSC (orange). 
	(a) Stochastic behaviors of the state $x_{t}$ for 100 samples. 
	The black rectangles and cross represent the obstacles and the goal, respectively. 
	(b) The objective function (\ref{eq: OF of ML-POSC non-LQG NE}), which is computed from 100 samples. 
	}
	\label{fig: non-LQG}
\end{center}
\end{figure}

In this subsection, we investigate the potential effectiveness of ML-POSC to a non-LQG problem by comparing it with the local LQG approximation of the conventional POSC \cite{li_iterative_2006,li_iterative_2007}.
We consider the state $x_{t}\in\mb{R}$ and the observation $y_{t}\in\mb{R}$, which evolve by the following SDEs: 
\begin{align}
	dx_{t}&=u_{t}dt+d\omega_{t},\label{eq: state SDE non-LQG NE}\\
	dy_{t}&=x_{t}dt+d\nu_{t},\label{eq: observation SDE non-LQG NE}
\end{align}
where $x_{0}$ obeys the Gaussian distribution $p_{0}(x_{0})=\mcal{N}(x_{0}|0,0.01)$, $y_{0}$ is an arbitrary real number, $\omega_{t}\in\mb{R}$ and $\nu_{t}\in\mb{R}$ are independent standard Wiener processes, $u_{t}=u(t,y_{0:t})\in\mb{R}$ is the control. 
The objective function to be minimized is given as follows: 
\begin{align}	
	J[u]:=\mb{E}\left[\int_{0}^{1}\left(Q(t,x_{t})+u_{t}^{2}\right)dt+10x_{1}^{2}\right], 
	\label{eq: OF of ML-POSC non-LQG NE}
\end{align} 
where 
\begin{align}
	Q(t,x):=\begin{cases}
		1000&(0.3\leq t\leq 0.6, 0.1\leq |x|\leq 2.0),\\
		0&(others).
	\end{cases}
	\label{eq: state cost function of non-LQG NE}
\end{align}
The cost function is large on the black rectangles in Fig. \ref{fig: non-LQG} (a), which represents the obstacles. 
In addition, the terminal cost function is the smallest on the black cross in Fig. \ref{fig: non-LQG} (a), which represents the desirable goal.
Therefore, the system should avoid the obstacles and reach the goal with the small control. 
Because the cost function is non-quadratic, it is a non-LQG problem, which cannot be solved exactly by the conventional POSC. 

In the local LQG approximation of the conventional POSC \cite{li_iterative_2006,li_iterative_2007}, the Zakai equation and the Bellman equation are locally approximated by the Kalman filter and the Riccati equation, respectively. 
Because the Bellman equation (functional differential equation) is reduced to the Riccati equation (ordinary differential equation), the local LQG approximation can be solved numerically even in the non-LQG problem. 

ML-POSC determines the control $u_{t}\in\mb{R}$ based on the memory $z_{t}\in\mb{R}$, i.e., $u_{t}=u(t,z_{t})$. 
The memory dynamics is formulated with the following SDE: 
\begin{align}
	dz_{t}&=dy_{t},\label{eq: memory SDE non-LQG NE}
\end{align}
where $p_{0}(z_{0})=\mcal{N}(z_{0}|0,0.01)$. 
For the sake of simplicity, the memory control is not considered. 
In ML-POSC, the Bellman equation (functional differential equation) is reduced to the HJB equation (partial differential equation) by employing the mathematical technique of the mean-field control theory. 
As a result, ML-POSC can be solved numerically even in the non-LQG problem. 

Fig. \ref{fig: non-LQG} is the numerical result comparing the local LQG approximation and ML-POSC. 
Because the local LQG approximation reduces the Bellman equation to the Riccati equation by ignoring non-LQG information, it cannot avoid the obstacles, which results in the higher objective function. 
In contrast, because ML-POSC reduces the Bellman equation to the HJB equation while maintaining non-LQG information, it can avoid the obstacles, which results in the lower objective function. 
Therefore, our numerical experiment shows that ML-POSC can be superior to the local LQG approximation.

\section{CONCLUSION}\label{sec: Conclusion}
In this work, we propose ML-POSC, which is the alternative theoretical framework to the conventional POSC. 
ML-POSC first formulates the finite-dimensional and stochastic memory dynamics explicitly, and then optimizes the memory dynamics considering the memory cost. 
As a result, unlike the conventional POSC, ML-POSC can consider memory limitation as well as incomplete information. 
Furthermore, because the optimal control function of ML-POSC is obtained by solving the system of HJB-FP equations, ML-POSC can be solved in practice even in a non-LQG problem. 
ML-POSC can generalize the LQG problem to include memory limitation. 
Because estimation and control are not clearly separated in the LQG problem with memory limitation, the Riccati equation is modified to the partially observable Riccati equation, which improves estimation as well as control. 
Furthermore, ML-POSC can provide a better result than the local LQG approximation in a non-LQG problem because ML-POSC reduces the Bellman equation while maintaining non-LQG information. 

ML-POSC is also effective to the state estimation problem, which is a part of the POSC problem. 
Although the state estimation problem can be solved in principle by the Zakai equation \cite{crisan_survey_2002,budhiraja_survey_2007,bain_fundamentals_2009}, it cannot be solved directly because the Zakai equation is infinite-dimensional. 
In order to resolve this problem, the particle filter is often used, which approximates the infinite-dimensional Zakai equation by  a finite number of particles \cite{crisan_survey_2002,budhiraja_survey_2007,bain_fundamentals_2009}. 
However, because the performance of the particle filter is guaranteed only in the limit of a large number of particles, the particle filter may not be practical in the case where the available memory size is severely limited.  
Furthermore, the particle filter cannot take the memory noise and cost into account. 
ML-POSC resolves these problems because it can optimize the state estimation under memory limitation. 

ML-POSC may be extended from a single-agent system to a multi-agent system. 
POSC of a multi-agent system is called decentralized stochastic control (DSC)  \cite{nayyar_decentralized_2013,charalambous_centralized_2017,charalambous_centralized_2018}, which consists of a system and multiple controllers. 
In DSC, each controller needs to estimate the controls of the other controllers as well as the state of the system, which is essentially different from the conventional POSC. 
Because the estimation among the controllers is generally intractable, the conventional POSC approach cannot be straightforwardly extended to DSC. 
In contrast, ML-POSC compresses the observation history into the finite-dimensional memory, which simplifies the estimation among the controllers. 
Therefore, ML-POSC may provide an effective approach to DSC. 
Actually, the finite-state controller, whose idea is similar with ML-POSC, plays a key role in extending POMDP from a single-agent system to a multi-agent system \cite{oliehoek_concise_2016,bernstein_bounded_2005,bernstein_policy_2009,amato_optimizing_2007,amato_finite-state_2010,tottori_forward_2021}. 
ML-POSC may also be extended to a multi-agent system in the similar way as the finite-state controller. 

ML-POSC can be naturally extended to the mean-field control setting \cite{bensoussan_mean_2013,carmona_probabilistic_2018,carmona_probabilistic_2018-1} because ML-POSC is solved based on the mean-field control theory. 
Therefore, ML-POSC can be applied to an infinite number of homogeneous agents. 
Furthermore, ML-POSC can also be extended to a risk-sensitive setting because it is a special case of the mean-field control setting \cite{bensoussan_mean_2013,carmona_probabilistic_2018,carmona_probabilistic_2018-1}. 
Therefore, ML-POSC can consider the variance of the cost as well as its expectation. 

In order to solve ML-POSC with a high-dimensional state and memory, more efficient algorithms are needed. 
In the mean-field game and control, neural network-based algorithms have recently been proposed, which can solve high-dimensional problems efficiently \cite{ruthotto_machine_2020,lin_alternating_2021}. 
By extending these algorithms, high-dimensional ML-POSC may be solved efficiently. 
Furthermore, unlike the mean-field game and control, the coupling of HJB-FP equations is limited to the optimal control function in ML-POSC. 
By exploiting this property, more efficient algorithms may be proposed in ML-POSC \cite{tottori_notitle_2022}.

\section*{APPENDIX}
\subsection{Proof of Theorem \ref{theo: optimal control of GML-POSC based on Bellman eq}}
We define the value function $V(t,p)$ as follows: 
\begin{align}
	V(t,p):=\min_{u_{t:T}}\left[\int_{t}^{T}\bar{f}(t,p_{\tau},u_{\tau})d\tau+\bar{g}_{T}(p_{T})\right],
\end{align}
where $\{p_{\tau}|\tau\in[t,T]\}$ is the solution of the FP equation (\ref{eq: FP eq}) where $p_{t}=p$. 
$V(t,p)$ can be calculated as follows: 
\begin{align}
	V(t,p)
	&=\min_{u}\left[\bar{f}(t,p,u)dt+V(t+dt,p+\mcal{L}^{\dagger}pdt)\right]\nonumber\\
	&=\min_{u}\left[\bar{f}(t,p,u)dt+V(t,p)+\frac{\pl V(t,p)}{\pl t}dt+\left(\int \frac{\delta V(t,p)}{\delta p}(s)\mcal{L}^{\dagger}p(s)ds\right)dt\right]. 
\end{align}
Rearranging the above equation, the following equation is obtained: 
\begin{align}
	-\frac{\pl V(t,p)}{\pl t}=\min_{u}\left[\bar{f}(t,p,u)+\int \frac{\delta V(t,p)}{\delta p}(s)\mcal{L}^{\dagger}p(s)ds\right].
\end{align}
Because 
\begin{align}
	\int \frac{\delta V(t,p)}{\delta p}(s)\mcal{L}^{\dagger}p(s)ds
	=\int p(s)\mcal{L}\frac{\delta V(t,p)}{\delta p}(s)ds, 
\end{align}
the following Bellman equation is obtained: 
\begin{align}
	-\frac{\pl V(t,p)}{\pl t}=\min_{u}\mb{E}_{p(s)}\left[H\left(t,s,u,\frac{\delta V(t,p)}{\delta p}(s)\right)\right].
	\label{eq: Bellman eq of GML-POSC tmp}
\end{align}

Because the control $u$ is the function of the memory $z$ in ML-POSC, 
the minimization by $u$ can be exchanged with the expectation by $p(z)$ as follows: 
\begin{align}
	-\frac{\pl V(t,p)}{\pl t}=\mb{E}_{p(z)}\left[\min_{u}\mb{E}_{p(x|z)}\left[H\left(t,s,u,\frac{\delta V(t,p)}{\delta p}(s)\right)\right]\right].
	\label{eq: Bellman eq of GML-POSC exchanged version}
\end{align}
Because the optimal control function is given by the right-hand side of the Bellman equation (\ref{eq: Bellman eq of GML-POSC exchanged version}) \cite{yong_stochastic_1999}, the optimal control function is given by
\begin{align}
	u^{*}(t,z,p)=\argmin_{u}\mb{E}_{p(x|z)}\left[H\left(t,s,u,\frac{\delta V(t,p)}{\delta p}(s)\right)\right]. 
\end{align}
Because the FP equation (\ref{eq: FP eq}) is deterministic, the optimal control function is given by $u^{*}(t,z)=u^{*}(t,z,p_{t})$.

\subsection{Proof of Theorem \ref{theo: optimal control of GML-POSC}}
We first define 
\begin{align}
	W(t,p,s):=\frac{\delta V(t,p)}{\delta p}(s), 
\end{align}
which satisfies $W(T,p,s)=g(s)$. 
Differentiating the Bellman equation (\ref{eq: Bellman eq of GML-POSC}) with respect to $p$, 
the following equation is obtained: 
\begin{align}
	-\frac{\pl W(t,p,s)}{\pl t}=H\left(t,s,u^{*},W\right)
	+\mb{E}_{p(s')}\left[\mcal{L}\frac{\delta W(t,p,s')}{\delta p}(s)\right]. 
\end{align}
Because 
\begin{align}
	\int\mcal{L}\frac{\delta W(t,p,s)}{\delta p}(s')p(s')ds'
	=\int\frac{\delta W(t,p,s)}{\delta p}(s')\mcal{L}^{\dagger}p(s')ds',
\end{align}
the following equation is obtained: 
\begin{align}
	-\frac{\pl W(t,p,s)}{\pl t}
	&=H\left(t,s,u^{*},W\right)
	+\int\frac{\delta W(t,p,s)}{\delta p}(s')\mcal{L}^{\dagger}p(s')ds'. 
	\label{eq: Master eq}
\end{align} 

We then define 
\begin{align}
	w(t,s):=W(t,p_{t},s), 
\end{align}
where $p_{t}$ is the solution of the FP equation (\ref{eq: FP eq}). 
The time derivative of $w(t,s)$ can be calculated as follows: 
\begin{align}
	\frac{\pl w(t,s)}{\pl t}
	&=\frac{\pl W(t,p_{t},s)}{\pl t}+\int\frac{\delta W(t,p_{t},s)}{\delta p}(s')\frac{\pl p_{t}(s')}{\pl t}ds'\nonumber\\
	&=\frac{\pl W(t,p_{t},s)}{\pl t}+\int\frac{\delta W(t,p_{t},s)}{\delta p}(s')\mcal{L}^{\dagger}p_{t}(s')ds'. 
	\label{eq: the partial derivative of w(t,s)}
\end{align}
By substituting (\ref{eq: Master eq}) into (\ref{eq: the partial derivative of w(t,s)}), the HJB equation (\ref{eq: HJB eq}) is obtained. 

\subsection{Proof of Theorem \ref{theo: optimal control of COSC}}
From the proof of Theorem \ref{theo: optimal control of GML-POSC based on Bellman eq} (Appendix A), the Bellman equation (\ref{eq: Bellman eq of GML-POSC tmp}) is obtained. 
Because the control $u$ is the function of the extended state $s$ in the COSC of the extended state, 
the minimization by $u$ can be exchanged with the expectation by $p(s)$ as follows: 
\begin{align}
	-\frac{\pl V(t,p)}{\pl t}=\mb{E}_{p(s)}\left[\min_{u}H\left(t,s,u,\frac{\delta V(t,p)}{\delta p}(s)\right)\right].
	\label{eq: Bellman eq of COSC exchanged version}
\end{align}
Because the optimal control function is given by the right-hand side of the Bellman equation (\ref{eq: Bellman eq of COSC exchanged version}) \cite{yong_stochastic_1999}, the optimal control function is given by
\begin{align}
	u^{*}(t,s,p)=\argmin_{u}H\left(t,s,u,\frac{\delta V(t,p)}{\delta p}(s)\right). 
\end{align}
Because the FP equation (\ref{eq: FP eq}) is deterministic, the optimal control function is given by $u^{*}(t,s)=u^{*}(t,s,p_{t})$.
The rest of the proof is the same with the proof of Theorem \ref{theo: optimal control of GML-POSC} (Appendix B).

\subsection{Proof of Theorem \ref{theo: optimal control of ML-POSC in LQG-POSC}}
From Theorem \ref{theo: optimal control of GML-POSC}, the optimal control functions $u^{*}$, $v^{*}$, and $\kappa^{*}$ are given by the minimization of the conditional expected Hamiltonian as follows: 
\begin{align}
	u^{*}(t,z),v^{*}(t,z),\kappa^{*}(t,z)=\argmin_{u,v,\kappa}\mb{E}_{p_{t}(x|z)}\left[H\left(t,s,u,v,\kappa,w\right)\right].
\end{align}
In the LQG problem of the conventional POSC, the Hamiltonian (\ref{eq: Hamiltonian}) is given by
\begin{align}
	H(t,s,u,v,\kappa,w)
	&=x^{T}Qx+u^{T}Ru+\left(\frac{\pl w(t,s)}{\pl x}\right)^{T}\left(Ax+Bu\right)
	+\left(\frac{\pl w(t,s)}{\pl z}\right)^{T}\left(v+\kappa Hx\right)\nonumber\\
	&\ \ \ +\frac{1}{2}\tr\left\{\frac{\pl}{\pl x}\left(\frac{\pl w(t,s)}{\pl x}\right)^{T}\sigma\sigma^{T}\right\}
	+\frac{1}{2}\tr\left\{\frac{\pl}{\pl z}\left(\frac{\pl w(t,s)}{\pl z}\right)^{T}\kappa\gamma\gamma^{T}\kappa^{T}\right\}.
\end{align}
From 
\begin{align}
	\frac{\pl\mb{E}_{p_{t}(x|z)}\left[H\right]}{\pl u}&=2Ru+B^{T}\mb{E}_{p_{t}(x|z)}\left[\frac{\pl w(t,s)}{\pl x}\right],\\
	\frac{\pl\mb{E}_{p_{t}(x|z)}\left[H\right]}{\pl v}&=\mb{E}_{p_{t}(x|z)}\left[\frac{\pl w(t,s)}{\pl z}\right],\\
	\frac{\pl\mb{E}_{p_{t}(x|z)}\left[H\right]}{\pl \kappa}&=\mb{E}_{p_{t}(x|z)}\left[\frac{\pl w(t,s)}{\pl z}x^{T}\right]H^{T}
	+\mb{E}_{p_{t}(x|z)}\left[\frac{\pl}{\pl z}\left(\frac{\pl w(t,s)}{\pl z}\right)^{T}\right]\kappa\gamma\gamma^{T},
\end{align}
the optimal control functions are given by
\begin{align}
	u^{*}(t,z)&=-\frac{1}{2}R^{-1}B^{T}\mb{E}_{p_{t}(x|z)}\left[\frac{\pl w(t,s)}{\pl x}\right],\\
	v^{*}(t,z)&=\begin{cases}
		+\infty&\mb{E}_{p_{t}(x|z)}\left[\frac{\pl w(t,s)}{\pl z}\right]<0,\\
		arbitrary&\mb{E}_{p_{t}(x|z)}\left[\frac{\pl w(t,s)}{\pl z}\right]=0,\\
		-\infty&\mb{E}_{p_{t}(x|z)}\left[\frac{\pl w(t,s)}{\pl z}\right]>0,\\
	\end{cases}\\
	\kappa^{*}(t,z)&=-\left(\mb{E}_{p_{t}(x|z)}\left[\frac{\pl}{\pl z}\left(\frac{\pl w(t,s)}{\pl z}\right)^{T}\right]\right)^{-1}\mb{E}_{p_{t}(x|z)}\left[\frac{\pl w(t,s)}{\pl z}x^{T}\right]H^{T}(\gamma\gamma^{T})^{-1}.
\end{align}

We assume that $p_{t}(s)$ is given by the Gaussian distribution
\begin{align}
	p_{t}(s)=\mcal{N}(s|\mu(t),\Sigma(t)),
	\label{eq: assumption of FP LQG-POSC}
\end{align}
and $w(t,s)$ is given by the quadratic function 
\begin{align}
	w(t,s)&=x^{T}\Psi(t)x+(x-z)^{T}\Phi(t)(x-z)+\beta(t).
	\label{eq: assumption of HJB LQG-POSC}
\end{align}
From the initial condition of the FP equation,   
\begin{align}
	\mu(0)
	&=\left(\begin{array}{c}
		\mu_{x}(0)\\ \mu_{z}(0)\\
	\end{array}\right)
	=\left(\begin{array}{c}
		\mu_{x,0}\\ \mu_{x,0}\\
	\end{array}\right),\\
	\Sigma(0)
	&=\left(\begin{array}{cc}
		\Sigma_{xx}(0)&\Sigma_{xz}(0)\\ 
		\Sigma_{zx}(0)&\Sigma_{zz}(0)\\
	\end{array}\right)
	=\left(\begin{array}{cc}
		\Sigma_{xx,0}&O\\ 
		O&O\\ 
	\end{array}\right)
\end{align}
are satisfied. 
From the terminal condition of the HJB equation,  $\Psi(T)=P$, $\Phi(T)=O$, and $\beta(T)=0$ are satisfied. 
In this case, $u^{*}(t,z)$, $\mb{E}_{p_{t}(x|z)}[\pl w(t,s)/\pl z]$, and $\kappa^{*}(t,z)$ can be calculated as follows: 
\begin{align}
	&u^{*}(t,z)=-R^{-1}B^{T}\left(\left(\Psi+\Phi\right)\mu_{x|z}-\Phi z\right),\\
	&\mb{E}_{p_{t}(x|z)}\left[\frac{\pl w(t,s)}{\pl z}\right]=2\Phi\left(z-\mu_{x|z}\right),\\
	&\kappa^{*}(t,z)=\left(\Sigma_{x|z}+\left(\mu_{x|z}-z\right)\mu_{x|z}^{T}\right)H^{T}(\gamma\gamma^{T})^{-1}.
\end{align}

We then assume that the following equations are satisfied: 
\begin{align}
	\mu_{x}&=\mu_{z},\\
	\Sigma_{zz}&=\Sigma_{xz}.
\end{align}
In this case, $\mu_{x|z}$, $\Sigma_{x|z}$, $u^{*}(t,z)$, $\mb{E}_{p_{t}(x|z)}[\pl w(t,s)/\pl z]$, and $\kappa^{*}(t,z)$ can be calculated as follows: 
\begin{align}
	&\mu_{x|z}=z,\\
	&\Sigma_{x|z}=\Sigma_{xx}-\Sigma_{zz},\\
	&u^{*}(t,z)=-R^{-1}B^{T}\Psi z,\\
	&\mb{E}_{p_{t}(x|z)}\left[\frac{\pl w(t,s)}{\pl z}\right]=0,\\
	&\kappa^{*}(t,z)=\Sigma_{x|z}H^{T}(\gamma\gamma^{T})^{-1}.
\end{align}

Because $v^{*}(t,z)$ is arbitrary when $\mb{E}_{p_{t}(x|z)}[\pl w(t,s)/\pl z]=0$, we formulate $v^{*}(t,z)$ with the following equation: 
\begin{align}
	v^{*}(t,z)=\left(A-BR^{-1}B^{T}\Psi-\Sigma_{x|z}H^{T}(\gamma\gamma^{T})^{-1}H\right)z. 
\end{align}
In this case, the extended state SDE is given by the following equation: 
\begin{align}
	ds_{t}=\tilde{A}(t)s_{t}dt+\tilde{\sigma}(t)d\tilde{\omega}_{t}, 
	\label{eq: extended state SDE LQG-POSC}
\end{align}
where $p_{0}(s)=\mcal{N}(s|\mu(0),\Sigma(0))$, and 
\begin{align}
	\tilde{A}&:=\left(\begin{array}{cc}
		A&-BR^{-1}B^{T}\Psi\\
		\Sigma_{x|z}H^{T}(\gamma\gamma^{T})^{-1}H&A-BR^{-1}B^{T}\Psi-\Sigma_{x|z}H^{T}(\gamma\gamma^{T})^{-1}H\\
	\end{array}\right),\\
	\tilde{\sigma}&:=\left(\begin{array}{cc}
		\sigma&O\\
		O&\Sigma_{x|z}H^{T}(\gamma\gamma^{T})^{-1}\gamma\\
	\end{array}\right),\ 
	d\tilde{\omega}_{t}:=\left(\begin{array}{c}
		d\omega_{t}\\ d\nu_{t}\\
	\end{array}\right).
\end{align}
Because the drift and diffusion coefficients of (\ref{eq: extended state SDE LQG-POSC}) are linear and constant with respect to $s$, respectively, 
$p_{t}(s)$ becomes the Gaussian distribution, which is consistent with our assumption (\ref{eq: assumption of FP LQG-POSC}). 
$\mu(t)$ and $\Sigma(t)$ evolve by the following ordinary differential equations: 
\begin{align}
	\frac{d\mu}{dt}&=\tilde{A}\mu,\\
	\frac{d\Sigma}{dt}&=\tilde{\sigma}\tilde{\sigma}^{T}+\tilde{A}\Sigma+\Sigma\tilde{A}^{T}.
\end{align}
If $\mu_{x}=\mu_{z}$ and $\Sigma_{zz}=\Sigma_{xz}$ are satisfied, $d\mu_{x}/dt=d\mu_{z}/dt$ and $d\Sigma_{xz}/dt=d\Sigma_{zz}/dt$ are satisfied, respectively, which are consistent with our assumptions of $\mu_{x}=\mu_{z}$ and $\Sigma_{zz}=\Sigma_{xz}$. 

From $v^{*}$ and $\kappa^{*}$, the dynamics of $\mu_{x|z}(t,z_{t})=z_{t}$ is given by
\begin{align}
	dz_{t}=\left(A-BR^{-1}B^{T}\Psi\right)z_{t}dt+\Sigma_{x|z}H^{T}(\gamma\gamma^{T})^{-1}\left(dy_{t}-Hz_{t}dt\right),\label{eq: Kalman filter mu in ML-POSC tmp} 
\end{align}
where $z_{0}=\mu_{x,0}$. 
From $d\Sigma_{xx}/dt$ and $d\Sigma_{zz}/dt$, the dynamics of $\Sigma_{x|z}=\Sigma_{xx}-\Sigma_{zz}$ is given by
\begin{align}
	\frac{d\Sigma_{x|z}}{dt}=\sigma\sigma^{T}+A\Sigma_{x|z}+\Sigma_{x|z}A^{T}-\Sigma_{x|z}H^{T}(\gamma\gamma^{T})^{-1}H\Sigma_{x|z},\label{eq: Kalman filter Sigma in ML-POSC tmp}
\end{align}
where $\Sigma_{x|z}(0)=\Sigma_{xx,0}$. 
We note that (\ref{eq: Kalman filter mu in ML-POSC tmp}), (\ref{eq: Kalman filter Sigma in ML-POSC tmp}) corresponds to the Kalman filter (\ref{eq: Kalman filter mu}), (\ref{eq: Kalman filter Sigma}). 

By substituting $w(t,s)$, $u^{*}(t,z)$, $v^{*}(t,z)$, and $\kappa^{*}(t,z)$ into the HJB equation (\ref{eq: HJB eq}), 
we obtain the following ordinary differential equations: 
\begin{align}
	&-\frac{d\Psi}{dt}=Q+A^{T}\Psi+\Psi A-\Psi BR^{-1}B^{T}\Psi,\label{eq: ODE of Psi LQG-POSC tmp}\\
	&-\frac{d\Phi}{dt}=\left(A-\Sigma_{x|z}H^{T}(\gamma\gamma^{T})^{-1}H\right)^{T}\Phi+\Phi\left(A-\Sigma_{x|z}H^{T}(\gamma\gamma^{T})^{-1}H\right)+\Psi BR^{-1}B^{T}\Psi,\label{eq: ODE of Pizz LQG-POSC tmp}\\
	&-\frac{d\beta}{dt}=\tr\left\{\left(\Psi+\Phi\right)\sigma\sigma^{T}\right\}+\tr\left\{\Phi\Sigma_{x|z}H^{T}(\gamma\gamma^{T})^{-1}H\Sigma_{x|z}\right\},\label{eq: ODE of beta LQG-POSC tmp}
\end{align}
where $\Psi(T)=P$, $\Phi(T)=O$, and $\beta(T)=0$. 
If $\Psi(t)$, $\Phi(t)$, and $\beta(t)$ satisfy (\ref{eq: ODE of Psi LQG-POSC tmp}), (\ref{eq: ODE of Pizz LQG-POSC tmp}), and (\ref{eq: ODE of beta LQG-POSC tmp}), respectively, the HJB equation (\ref{eq: HJB eq}) is satisfied, which is consistent with our assumption (\ref{eq: assumption of HJB LQG-POSC}). 
We note that (\ref{eq: ODE of Psi LQG-POSC tmp}) corresponds to the Riccati equation (\ref{eq: ODE of Psi LQG-POSC}). 

\subsection{Proof of Theorem \ref{theo: optimal control of LQG in ML-POSC}}
From Theorem \ref{theo: optimal control of GML-POSC}, the optimal control function $u^{*}$ is given by the minimization of the conditional expected Hamiltonian as follows: 
\begin{align}
	u^{*}(t,z)=\argmin_{u}\mb{E}_{p_{t}(x|z)}\left[H\left(t,s,u,w\right)\right].
\end{align}
In the LQG problem with memory limitation, the Hamiltonian (\ref{eq: Hamiltonian}) is given as follows: 
\begin{align}
	H(t,s,u,w)
	&=s^{T}Qs+u^{T}Ru+\left(\frac{\pl w(t,s)}{\pl s}\right)^{T}\left(As+Bu\right)
	+\frac{1}{2}\tr\left\{\frac{\pl}{\pl s}\left(\frac{\pl w(t,s)}{\pl s}\right)^{T}\sigma\sigma^{T}\right\}.
\end{align}
From 
\begin{align}
	\frac{\pl \mb{E}_{p_{t}(x|z)}\left[H(t,s,u,w)\right]}{\pl u}=2Ru+B^{T}\mb{E}_{p_{t}(x|z)}\left[\frac{\pl w(t,s)}{\pl s}\right],
\end{align}
the optimal control function is given by
\begin{align}
	u^{*}(t,z)
	&=-\frac{1}{2}R^{-1}B^{T}\mb{E}_{p_{t}(x|z)}\left[\frac{\pl w(t,s)}{\pl s}\right].
	\label{eq: optimal control of LQG ver1}
\end{align}

We assume that $p_{t}(s)$ is given by the Gaussian distribution
\begin{align}
	p_{t}(s)=\mcal{N}(s|\mu(t),\Sigma(t)), 
	\label{eq: assumption of FP LQG}
\end{align}
and $w(t,s)$ is given by the quadratic function 
\begin{align}
	w(t,s)=s^{T}\Pi(t)s+\alpha^{T}(t)s+\beta(t).
	\label{eq: assumption of HJB LQG}
\end{align}
From the initial condition of the FP equation, $\mu(0)=\mu_{0}$ and $\Sigma(0)=\Sigma_{0}$ are satisfied. 
From the terminal condition of the HJB equation, $\Pi(T)=P$, $\alpha(T)=0$, and $\beta(T)=0$ are satisfied. 
In this case, the optimal control function (\ref{eq: optimal control of LQG ver1}) can be calculated as follows: 
\begin{align}
	u^{*}(t,z)=-\frac{1}{2}R^{-1}B^{T}\left(2\Pi K\hat{s}+2\Pi\mu+\alpha\right), 
	\label{eq: optimal control of LQG ver3}
\end{align}
where we use (\ref{eq: conditional mean vector of LQG}).  
Because the optimal control function (\ref{eq: optimal control of LQG ver3}) is linear with respect to $\hat{s}$, 
$p_{t}(s)$ is the Gaussian distribution, which is consistent with our assumption (\ref{eq: assumption of FP LQG}). 

By substituting (\ref{eq: assumption of HJB LQG}) and (\ref{eq: optimal control of LQG ver3}) into the HJB equation (\ref{eq: HJB eq}), 
we obtain the following ordinary differential equations: 
\begin{align}
	-\frac{d\Pi}{dt}&=Q+A^{T}\Pi+\Pi A-\Pi BR^{-1}B^{T}\Pi+\mcal{Q},\label{eq: ODE of Pi pre}\\
	-\frac{d\alpha}{dt}&=(A-BR^{-1}B^{T}\Pi)^{T}\alpha-2\mcal{Q}\mu,\label{eq: ODE of alpha}\\
	-\frac{d\beta}{dt}&=\tr(\Pi\sigma\sigma^{T})-\frac{1}{4}\alpha^{T}BR^{-1}B^{T}\alpha+\mu^{T}\mcal{Q}\mu,\label{eq: ODE of beta}
\end{align}
where $\mcal{Q}:=(I-K)^{T}\Pi BR^{-1}B^{T}\Pi (I-K)$. 
If $\Pi(t)$, $\alpha(t)$, and $\beta(t)$ satisfy (\ref{eq: ODE of Pi pre}), (\ref{eq: ODE of alpha}), and (\ref{eq: ODE of beta}), respectively, the HJB equation (\ref{eq: HJB eq}) is satisfied, which is consistent with our assumption (\ref{eq: assumption of HJB LQG}). 

By defining $\Upsilon(t)$ by $\alpha(t)=2\Upsilon(t)\mu(t)$, the optimal control function (\ref{eq: optimal control of LQG ver3}) can be calculated as follows: 
\begin{align}
	u^{*}(t,z)=-R^{-1}B^{T}\left(\Pi K\hat{s}+(\Pi+\Upsilon)\mu\right).
	\label{eq: optimal control of LQG ver4}
\end{align}
In this case, $\mu(t)$ obeys the following ordinary differential equation: 
\begin{align}
	\frac{d\mu}{dt}&=\left(A-BR^{-1}B^{T}(\Pi+\Upsilon)\right)\mu.\label{eq: ODE of mu tmp}
\end{align}
From $\alpha(t)=2\Upsilon(t)\mu(t)$, (\ref{eq: ODE of alpha}) and (\ref{eq: ODE of mu tmp}), $\Upsilon(t)$ obeys the following ordinary differential equation: 
\begin{align}
	-\frac{d\Upsilon}{dt}&=\left(A-BR^{-1}B^{T}\Pi\right)^{T}\Upsilon+\Upsilon\left(A-BR^{-1}B^{T}\Pi\right)-\Upsilon BR^{-1}B^{T}\Upsilon-\mcal{Q},\label{eq: ODE of Upsilon}
\end{align}
where $\Upsilon(T)=O$. 

By defining $\Psi(t):=\Pi(t)+\Upsilon(t)$, the optimal control function (\ref{eq: optimal control of LQG ver4}) can be calculated as follows: 
\begin{align}
	u^{*}(t,z)=-R^{-1}B^{T}\left(\Pi K\hat{s}+\Psi\mu\right).
	\label{eq: optimal control of LQG ver5}
\end{align}
From $\Psi(t)=\Pi(t)+\Upsilon(t)$, (\ref{eq: ODE of Pi pre}) and (\ref{eq: ODE of Upsilon}), $\Psi(t)$ obeys the following ordinary differential equation: 
\begin{align}
	-\frac{d\Psi}{dt}&=Q+A^{T}\Psi+\Psi A -\Psi BR^{-1}B^{T}\Psi,\label{eq: ODE of Psi tmp}
\end{align}
where $\Psi(T)=O$. 
Therefore, the optimal control function (\ref{eq: optimal control of LQG}) is obtained. 

\section*{ACKNOWLEDGMENT}
We thank K. Kashima and K. Ito for useful discussions. 
The first author received a JSPS Research Fellowship (Grant No. 21J20436). 
This work was supported by JSPS KAKENHI (Grant No. 19H05799) and JST CREST (Grant No. JPMJCR2011).

\bibliographystyle{ieeetr}
\bibliography{220321_ML-POSC_ref}
\end{document}